\documentclass[11pt]{amsart}
\usepackage[babel]{csquotes}
\usepackage{esint}
\usepackage{enumitem}
\usepackage{amsmath,amsthm,amssymb,mathrsfs,amsfonts,verbatim,enumitem,color,leftidx,mathtools}
\usepackage{mathabx}
\usepackage{etoolbox} 
\usepackage{bbm}
\usepackage[all,tips]{xy}
\usepackage{graphicx,ifpdf}
\usepackage{stmaryrd}
\usepackage{amssymb}
\usepackage{marginnote}
\usepackage{dsfont}

\ifpdf
\fi
\usepackage[colorlinks]{hyperref}
\hypersetup{
linkcolor=blue,          
citecolor=green,        
}

\usepackage{tikz}

\newtheorem{thm}{Theorem}[section]
\newtheorem{lem}[thm]{Lemma}

\newtheorem{prop}[thm]{Proposition}

\theoremstyle{definition}
\newtheorem{defi}[thm]{Definition}

\newtheorem{remark}[thm]{Remark}

\theoremstyle{remark}

\numberwithin{equation}{section}
\definecolor{esperance}{rgb}{0.0,0.5,0.0}

\newcommand{\ba}{\mathbf{a}}
\newcommand{\bb}{\mathbf{b}}

\newcommand{\bd}{\mathbf{d}}
\newcommand{\be}{\mathbf{e}}

\newcommand{\bm}{\mathbf{m}}
\newcommand{\bn}{\mathbf{n}}

\newcommand{\bv}{\mathbf{v}}

\newcommand{\del}{\delta}
\newcommand{\Del}{\Delta}

\newcommand{\eps}{\epsilon}

\newcommand{\sig}{\sigma}

\newcommand{\Om}{\Omega}

\newcommand{\cE}{\mathcal{E}}

\newcommand{\cH}{\mathcal{H}}

\newcommand{\cK}{\mathcal{K}}
\newcommand{\cL}{\mathcal{L}}

\newcommand{\cO}{\mathcal{O}}
\newcommand{\cP}{\mathcal{P}}
\newcommand{\cQ}{\mathcal{Q}}
\newcommand{\cR}{\mathcal{R}}

\newcommand{\cT}{\mathcal{T}}

\newcommand{\cW}{\mathcal{W}}

\newcommand{\bR}{\mathbb{R}}
\newcommand{\bZ}{\mathbb{Z}}
\newcommand{\bQ}{\mathbb{Q}}

\newcommand{\bN}{\mathbb{N}}

\newcommand{\SL}{\operatorname{SL}}

\newcommand*{\transp}[2][-1mu]{\ensuremath{\mskip1mu\prescript{\smash{\mathrm t\mkern#1}}{}{\mathstrut#2}}}

\newcommand\set[1]{\left\{#1\right\}}

\newcommand\tb[1]{\textbf{#1}}

\newcommand{\onto}{\xymatrix{\ar@{>>}[r]&}}

\newcommand{\eq}[1]
{
\begin{equation*}
{#1}
\end{equation*}
}
\newcommand{\eqlabel}[2]
{
\begin{equation}
{#2}\label{#1}
\end{equation}
}

\makeatletter
\newcommand*{\rom}[1]{\expandafter\@slowromancap\romannumeral #1@}
\makeatother

\begin{document}

\title[Moments of Margulis functions]{Moments of Margulis functions and indefinite ternary quadratic forms}
\author{Wooyeon Kim}
\address{Wooyeon Kim. Department of Mathematics, ETH Z\"{u}rich, 
{\it wooyeon.kim@math.ethz.ch}}

\thanks{}



\def\thefootnote{}
\def\thefootnote{\arabic{footnote}}
\setcounter{footnote}{0}

\begin{abstract}
In this paper, we prove a quantitative version of the Oppenheim conjecture for indefinite ternary quadratic forms: for any indefinite irrational ternary quadratic form $Q$ that is not extremely well approxiable by rational forms, and for $a<b$ the number of integral vectors of norm at most $T$ satisfying $a<Q(v)<b$ is asymptotically equivalent to $\big(\mathsf{C}_Q(b-a)+\mathsf{I}_{Q}(a,b)\big)T$ as $T$ tends to infinity, where the constant $\mathsf{C}_Q>0$ depends only on $Q$, and the term $\mathsf{I}_{Q}(a,b)T$ accounts for the contribution from rational isotropic lines and degenerate planes.

The main technical ingredient is a uniform bound for the $\lambda$-moment of the Margulis $\alpha$-function along expanding translates of a unipotent orbit in $\SL_3(\bR)/\SL_3(\bZ)$, for some $\lambda>1$. To establish this, we introduce a new height function $\widetilde{\alpha}$ on the space of lattices, which captures the failure of the classical Margulis inequality. This moment bound implies equidistribution of such translates with respect to a class of unbounded test functions, including the Siegel transform.

\end{abstract}

\maketitle\section{Introduction}
 The Oppenheim conjecture, proved by Margulis in 1986, states that for a non-degenerate indefinite irrational quadratic form $Q$ in $n\geq 3$ variables, the image set $Q(\bZ^d)$ of integral vectors is a dense subset of $\bR$ unless $Q$ is proportional to a form with rational coefficients. In this paper, we say that a quadratic form is \textit{rational} if it is proportional to a form with rational coefficients, and \textit{irrational} otherwise.

 A quantitative version of the Oppenheim conjecture was established in \cite{EMM98} and \cite{EMM05} for quadratic forms in $n\ge 4$ variables. For $T>0$ let $B(T)$ denote the ball of radius $T$ centered at zero, and for $a,b\in\bR$ with $a<b$ let $N_Q(a,b,T)$ be the cardinality of the set
 $$\set{v\in \bZ^n: v\in B(T) \textrm{ and }a<Q(v)<b}.$$
 Determining the asymptotic of $N_Q(a,b,T)$ is referred to as quantitative Oppenheim conjecture. In \cite{EMM98}, Eskin, Margulis, and Mozes showed that if $Q$ is an indefinite irrational quadratic form and the signature of $Q$ is not $(2,2)$ and $(2,1)$, then
\eqlabel{eq:EMM98asymptotic}{\lim_{T\to\infty}\frac{N_Q(a,b,T)}{T^{n-2}}= \mathsf{C}_Q(b-a),}
where $\mathsf{C}_Q>0$ depends only on $Q$.

If the signature of $Q$ is $(2,2)$ or $(2,1)$, then the universal formula \eqref{eq:EMM98asymptotic} is no longer true. In fact, there are irrational forms for which $N_Q(a,b,T_j)\gg T_j^{n-2}\log T_j$ along a subsequence $T_j\to\infty$, and one may obtain such quadratic forms by considering irrational forms which are very well approximated by rational quadratic forms. Thus, one needs to assume certain Diophantine conditions on quadratic forms of signature $(2,2)$ or $(2,1)$ to obtain an asymptotic formula like \eqref{eq:EMM98asymptotic}.

\begin{defi}
    For a quadratic form $Q$ we say that $Q$ is \textit{extremely well approximable} (EWA) if for any $c>0$ and $M>1$ there exists an integral quadratic form $Q'$ of the same signature such that
    $$\|Q-\rho Q'\|\leq  c\|Q'\|^{-M},$$ where $\rho=\rho(Q')$ is the constant satisfying $\det(\rho Q')=\det(Q)$. Here and hereafter, $\|\cdot\|$ for quadratic forms stands for the supremum of the coefficients.
\end{defi}

In \cite{EMM05} it was proved that if $Q$ is a quadratic form of signature $(2,2)$, which is not EWA, then
\eqlabel{eq:EMM05asymptotic}{\lim_{T\to\infty}\frac{\widetilde{N}_Q(a,b,T)}{T^{2}}= \mathsf{C}_Q(b-a),}
where $\widetilde{N}_Q(a,b,T)$ counts the points not contained in rational isotropic subspaces. 

The purpose of this paper is to establish the asymptotic behavior of $N_Q(a,b,T)$ for quadratic forms of signature $(2,1)$ that are not EWA.

\begin{defi}
    For a given indefinite ternary quadratic form $Q$, we say that a line $l\subset \bR^3$ through the origin is \textit{isotropic} if $Q(v_l)=0$ for a nonzero vector $v_l$ on $l$. We also say that a plane $P\subset \bR^3$ through the origin is \textit{degenerate} if the restriction $Q|_P$ on $P$ is the square of a linear form. 
\end{defi}

The main result of this paper is as follows.

\begin{thm}[Modified count]\label{thm:quantitativeoppenheim0}
    Let $Q$ be an indefinite ternary quadratic form that is not EWA. Then for any $a<b$ we have
    \eqlabel{eq:mainasymptotic}{\lim_{T\to\infty}\frac{\widetilde{N}_Q(a,b,T)}{T}=\mathsf{C}_Q(b-a),}
    where $\widetilde{N}_Q(a,b,T)$ denotes the number of points not lying on any rational isotropic lines or rational degenerate planes through the origin. The constant $\mathsf{C}_Q>0$ depends only on $Q$.
\end{thm}

Including the points lying on rational isotropic lines and rational degenerate planes, we have:

\begin{thm}[Complete count]\label{thm:quantitativeoppenheim}
    Let $Q$ be an indefinite ternary quadratic form that is not EWA. Then for any $a<b$ we have
    \eq{\lim_{T\to\infty}\frac{N_Q(a,b,T)}{T}=\mathsf{C}_Q(b-a)+\mathsf{I}_Q(a,b),}
    where $\mathsf{C}_Q>0$ depends only on $Q$, and the constant $\mathsf{I}_Q(a,b)\geq0$ depends only on $a$, $b$, and $Q$.
\end{thm}

\begin{remark}
    Note that $\mathsf{C}_Q>0$ is the constant satisfying the following asymptotic for any $a<b$:
     $$\lim_{T\to\infty}\frac{\operatorname{Vol}\big(\set{v\in \bR^n: v\in B(T) \textrm{ and }a<Q(v)<b}\big)}{T^{n-2}}=\mathsf{C}_Q(b-a).$$
    Furthermore the constant $\mathsf{C}_Q$ can be explicitly computed as follows:
\eqlabel{eq:CQexpression}{\mathsf{C}_Q=\int_{\cH\cap B(1)}\frac{d\sig}{\|\nabla Q\|},}
    where $\cH:=\set{(v_1,v_2,v_3)\in\bR^3:v_2^2-2v_1v_3=0}$, and $d\sig$ is the area element on $\cH$.
    
    In the asymptotic of $N_Q(a,b,T)$, the term $\mathsf{I}_Q(a,b)T$ accounts for the number of integral points lying on rational isotropic lines and degenerate planes. If there are no such rational subspaces through the origin, that is, if neither $Q=0$ nor $Q^*=0$ admits a nontrivial integral solution, then $\mathsf{I}_Q(a,b)=0$ for any $a<b$. Here, $Q^*$ is the dual quadratic form of $Q$ defined as $Q^*(v)=\transp{v}A^{-1}v$, where $A$ is the symmetric matrix associated to $Q$, i.e, $Q(v)=\transp{v}Av$. We note that $P$ is degenerate if and only if the line orthogonal to $P$ is isotropic for the dual form $Q^*$; that is, $Q^*(v_P)=0$ for a nonzero vector $v_P$ orthogonal to $P$ (see Lemma \ref{lem:degenerateplane}). Moreover, one can calculate the constant $\mathsf{I}_Q(a,b)$ explicitly in terms of the isotropic vectors (see \S\ref{subsec:exceptionalsubspaces}).

\end{remark}

\subsection{Moments of Margulis functions}
For the lower bound of $N_Q(a,b,T)$, Dani and Margulis showed in \cite{DM93} that for any irrational quadratic form of signature $(p,q)$ with $p\geq 2$ and $q\geq 1$ \eqlabel{eq:DMlowerbound}{\liminf_{T\to\infty}\frac{N_Q(a,b,T)}{T^{n-2}}\geq \mathsf{C}_Q(b-a),}
where $n=p+q$. The proof of the lower bound \eqref{eq:DMlowerbound} is based on Ratner's equidistribution theorem for the action of $\operatorname{SO}(p,q)$ on $\operatorname{SL}_n(\bR)/\operatorname{SL}_n(\bZ)$, and the linearization method developed in \cite{DM93}.

The proof of the upper bound of $N_Q(a,b,T)$ is more subtle since an equidistribution theorem for \textit{unbounded} test functions is required, hence one needs a sharp estimate for the quantitative recurrence to compact sets in the space of lattices. Margulis $\alpha$-function is a height function constructed in \cite{EMM98} for this purpose, and the moments of the height function integrated over $\operatorname{SO}(p,q)$-orbits are studied in \cite{EMM98} and \cite{EMM05}.

The construction of the $\alpha$-function on the space of lattices is as follows. Let $\Del$ be a lattice in $\bR^n$. A linear subspace $L$ of $\bR^n$ is called $\Del$-rational if $L\cap\Del$ is a lattice in $L$, and for a $\Del$-rational subspace $L$ we denote $d(L)$ the volume of $L\cap \Del$. For each $1\leq i\leq n-1$ let
\eqlabel{eq:alphaidef}{\alpha_i(\Del):=\sup\set{\frac{1}{d(L)}: L \textrm{ is a }\Del\textrm{-rational space of dimension }i},}
and define $\alpha(\Del):=\displaystyle\max_{1\leq i\leq n-1}\alpha_i(\Del)$.

From now on, let us focus on $n=3$. Let $G=\operatorname{SL}_3(\bR)$ and $\Gamma=\operatorname{SL}_3(\bZ)$. We denote by $X$ the space of unimodular lattices in $\bR^3$. The space $X$ is identified with $G/\Gamma$, since $\operatorname{Stab}_G(\bZ^3)=\Gamma$. Let $Q_0$ the standard indefinite ternary quadratic form defined by
$$Q_0(v_1,v_2,v_3)=v_2^2-2v_1v_3,$$
and let $H:=\operatorname{Stab}_G(Q_0)\simeq \operatorname{PSL}_2(\bR)$.

For a ternary quadratic form $Q$ and $g\in G$ we denote by $Q^g$ the quadratic form defined by $Q^g(v)=Q(gv)$. For any quadratic form $Q$ with $\operatorname{det}(Q)=1$ of signature $(2,1)$, there exists $g\in G$ such that $Q=Q_0^g$. Then let denote $\Delta_Q$ the lattice $g\bZ^3$, so that $Q(\bZ^3)=Q_0(\Delta_Q)$.

For a lattice $\Delta$ let $\Delta^*$ denote the dual lattice of $\Delta$. It holds that
\eq{\begin{aligned}
    \Delta^*&:=\set{v\in\bR^3: v\cdot w\in\bZ \textrm{ for all }w\in\Delta}\\&=\set{v\times w\in\bR^3: v,w\in\Delta}.
\end{aligned}}
Note that we can also view $\Delta^*$ as a lattice of $\bigwedge^2\bR^3$ via the canonical isomorphism $v\wedge w\mapsto v\times w$ from $\bigwedge^2\bR^3$ to $\bR^3$. We remark that the dual quadratic form $Q^*$ of $Q$ is the associated quadratic form to $\Delta_Q^*$. Indeed, $\Delta_Q^*=g^*\bZ^3$ and $Q^*=Q_0^{g^*}$, where $g^*:=\transp{g^{-1}}$. 

We have $\alpha_1(\Delta)=\sup\set{\|v\|^{-1}:v\in\Delta\setminus\set{0}}$ and by duality $\alpha_2(\Delta)=\alpha_1(\Delta^*)$. Now the Margulis height function can be written as \eqlabel{eq:alphadef}{\begin{aligned}\alpha(\Delta)&=\max\set{\alpha_1(\Delta),\alpha_2(\Delta)}=\max\set{\alpha_1(\Delta),\alpha_1(\Delta^*)}\\&=\sup\set{\|v\|^{-1}:v\in(\Delta\cup\Delta^*)\setminus\set{0}}.\end{aligned}}


For $t,r\in\bR$ let
$$a_t:=\left(\begin{matrix} e^t & & \\ & 1 & \\ & & e^{-t}\end{matrix}\right),\qquad  u_r:=\left(\begin{matrix} 1 & r & \frac{r^2}{2} \\ & 1 & r  \\ & & 1\end{matrix}\right).$$
The moments of the function $\alpha$ integrated over expanding translates of a unipotent trajectory were studied in \cite{EMM98}:
\eqlabel{MomentEMM}{\sup_{t>0}\int_{-1}^{1}\alpha(a_tu_r\Delta)^{\lambda}dr<\infty, \quad \sup_{t>0}\frac{1}{t}\int_{-1}^{1}\alpha(a_tu_r\Delta)dr<\infty}
for any $0<\lambda<1$ and $\Delta\in X$. 

\begin{remark}\label{rem:doublezero}
    The exponent $\lambda$ here cannot be improved beyond $1$ for every $\Delta$. For instance, if $\Delta$ contains a nonzero vector $v=(v_1,v_2,v_3)\in\bR^3$ such that $Q_0(v)=0$ and $|v_2|\leq |v_3|$, then $\int_{-1}^{1}\alpha(a_tu_r\Delta)^{\lambda}dr$ diverges as $t\to\infty$ for any $\lambda>1$. Indeed, if $Q_0(v)=0$ and $|v_2|\leq |v_3|$ then $\|a_tu_rv\|\leq e^{-t}\|v\|$ for any $|r+\frac{v_2}{v_3}|\leq e^{-t}$, hence for any $t>0$ we have
$$\int_{-1}^{1}\alpha(a_tu_r\Delta)^{\lambda}dr\geq \int_{|r+\frac{v_2}{v_3}|\leq e^{-t}}\|a_tu_rv\|^{-\lambda}dr\geq e^{(\lambda-1)t}\|v\|^{-\lambda}.$$
\end{remark}

In this paper, we improve the exponent $\lambda$ in \eqref{MomentEMM} beyond $1$ under suitable Diophantine assumptions on the initial point $\Delta\in X$ that accounts for the not-EWA condition on $Q$. To improve the exponent $\lambda$ in the moment of the Margulis function $\int_{-1}^{1}\alpha(a_tu_r\Del_Q)^{\lambda}dr$ beyond $1$, one needs to exclude the contribution of quasi-null rational subspaces in $\Del_Q$.

\begin{defi}[Quasi-null vectors]
    For $\eta>0$ and $M>1$ we say that a vector $v\in\bR^3$ is $(\eta,M)$-quasi-null if $|Q_0(v)|<\eta\|v\|^{-50M}$. We denote by $\cH_{\eta,M}$ the set of $(\eta,M)$-quasi-null vectors, i.e. $$\cH_{\eta,M}:=\set{v\in\bR^3: |Q_0(v)|<\eta\|v\|^{-50M}}.$$
\end{defi}

For $\eta>0$ and $M>1$ we define modified height functions $$\widehat{\alpha}_{1,\eta,M},\widehat{\alpha}_{2,\eta,M},\widehat{\alpha}_{\eta,M}:H \times X\to (0,\infty)$$ excluding quasi-null rational subspaces, by
\eq{\widehat{\alpha}_{1,\eta,M}(g;\Delta):= \sup\set{\|gv\|^{-1}:v\in \Delta\setminus \cH_{\eta,M}},\quad \widehat{\alpha}_{2,\eta,M}(g;\Delta):= \widehat{\alpha}_{1,\eta,M}(g^*;\Delta^*),}
\eq{\widehat{\alpha}_{\eta,M}(g;\Delta):= \max\set{\widehat{\alpha}_{1,\eta,M}(g;\Delta),\widehat{\alpha}_{2,\eta,M}(g;\Delta)},}
in the same spirit as in \cite{EMM05}.

\begin{defi}
    For $M>1$ we say that a ternary quadratic form $Q$ with $\operatorname{det}Q=1$ is of \textit{(Diophantine) type} $M$ if there exists $c>0$ such that
    $$\|Q-\rho Q'\|> c\|Q'\|^{-M}$$ for any nonzero integral ternary quadratic form $Q'$, where $\rho=(\operatorname{det}Q')^{-\frac{1}{3}}$.
\end{defi}

    Note that a ternary quadratic form $Q$ with $\operatorname{det}Q=1$ is EWA if and only if $Q$ is not of type $M$ for any $M>1$.


\begin{thm}[Main moment esitmates]\label{thm:mainthm}
    For any $M>1$ there exists a constant $\del=\del(M)>0$ such that the following holds. If $Q$ is of type $M$, then for any $\eta>0$ 
    \eqlabel{eq:unipotenthighermoment}{\sup_{t>0}\int_{-1}^{1}\widehat{\alpha}_{\eta,M}(a_tu_r;\Del_Q)^{1+\del}dr<\infty.}
\end{thm}

In fact, \eqref{eq:unipotenthighermoment} is equivalent to the following statement:
\eqlabel{eq:compacthighermoment}{\sup_{t>0}\int_{K}\widehat{\alpha}_{\eta,M}(a_tk;\Del_Q)^{1+\del}dk<\infty,}
where $K$ is the maximal compact subgroup of $H$. Then \eqref{eq:compacthighermoment} implies Theorem \ref{thm:quantitativeoppenheim0} and Theorem \ref{thm:quantitativeoppenheim}, by an argument identical to that used in \cite[\S3.4, \S3.5]{EMM98} to deduce \cite[Theorem 2.1]{EMM98} from \cite[Theorem 2.3]{EMM98}. The derivation of Theorem \ref{thm:quantitativeoppenheim0} and Theorem \ref{thm:quantitativeoppenheim} from Theorem \ref{thm:mainthm} is outlined in Subsection \S\ref{sec:momentfromcounting}.

\subsection{Margulis inequality}
The uniform boundedness \eqref{eq:compacthighermoment} of the $\lambda$-moment for some $\lambda>1$ forms the main technical core of this paper, playing a role analogous to that in \cite{EMM98, EMM05} for other signatures. For $(p,q)\neq (2,1),(2,2)$, the key ingredient in establishing an analog of \eqref{eq:compacthighermoment} was the so-called Margulis inequality developed in \cite{EMM98}. This inequality asserts that for any $\lambda<2$ and sufficiently large $t$, there exist constants $0<c<1$ and $B>1$ such that
$$\int_K\alpha(a_tk\Delta)^{\lambda}dk\leq c\alpha(\Delta)^{\lambda}+B$$
for every $\Delta\in \operatorname{SL}_d(\mathbb{R})/\operatorname{SL}_d(\mathbb{Z})$. Iterating this inequality yields the desired uniform boundedness of the $\lambda$-moment. Since \cite{EMM98}, Margulis inequalities have been established in broader contexts on homogeneous spaces, well beyond their original application in the quantitative Oppenheim conjecture. We refer the reader to \cite{EM22} for a historical overview and a comprehensive list of references.

However, as alluded to in Remark \ref{rem:doublezero}, for $(p,q)= (2,1),(2,2)$ the Margulis inequality holds only for exponents $\lambda<1$, which is insufficient to deduce the quantitative Oppenheim conjecture. To overcome this obstacle in the case $(p,q)=(2,2)$, \cite{EMM05} established the uniform boundedness of the moment for $\lambda=1.05$ directly, by employing a geometric argument using coverings by rectangles in $\mathbb{R}^2$. Nevertheless, it remained unclear how to generalize this method to more general settings in homogeneous dynamics, including the case $(p,q)=(2,1)$.

The approach taken in this paper is closer in spirit to the original strategy of \cite{EMM98} based on the Margulis inequality, rather than the geometric argument employed in \cite{EMM05}. The main novelty lies in the introduction of a new height function $\widetilde{\alpha}$, which incorporates the distance to the locus $Q_0=0$, the source of the failure of the Margulis inequality. This modified height function satisfies the desired Margulis inequality for most $\Delta$, and moreover, the exceptional set where the Margulis inequality fails can be characterized explicitly. The contribution from pieces of orbits passing through this exceptional set—a situation not encountered in \cite{EMM98}—is controlled using the finiteness of isotropic vectors implied by the irrationality of the quadratic form, together with an effective avoidance principle for periodic orbits. A more detailed outline of this strategy is provided in \S\ref{subsec:outline}. We expect that this method will offer a systematic approach to establishing equidistribution results with respect to unbounded test functions in broader settings.


\subsection{Related results}

Since Margulis proved the Oppenheim conjecture, various refinements and extensions of the Oppenheim conjecture have been studied. Recently, there has been increased interest in developing \textit{effective} versions of the Oppenheim conjecture. The basic question in this direction is as follows. Given a quadratic form $Q$ in $n$ variables and $\xi\in\bR$, how large is the size of a nontrivial solution $\bm\in\bZ^n$ to the Diphantine inequality $|Q(\bm)-\xi|<\eps$, in terms of $\eps^{-1}$? We remark that asymptotic formulas like \eqref{eq:EMM98asymptotic}, \eqref{eq:EMM05asymptotic}, and \eqref{eq:mainasymptotic} do not give information on the explicit size of $\bm$. For quadratic forms in $n\geq 5$ variables, Buterus, Götze, Hille, and Margulis established an effective version with a polynomial rate using Fourier analytic methods in \cite{BGHM22} (see also \cite{BG99,Got04}). For ternary quadratic forms, an effective version with a logarithmic rate was proved by Lindenstrauss and Margulis \cite{LM14}. Very recently, Lindenstrauss, Mohammadi, Wang, and Yang established an effective result with a polynomial rate in \cite{LMWY25}, building on their recent groundbreaking advances in effective equidistribution \cite{LM23,Yan25,LMW22}. Their result relies on a new effective equidistribution theorem, combined with the quantitative non-divergence estimate established in the present paper.

\subsection{Structure of the paper}
This paper is organized as follows. 

In Section~2, we introduce a notion of quasi-null vectors and define Margulis functions excluding the quasi-null vectors. Then we give an outline of the proof of Theorem~\ref{thm:mainthm}, the main technical result of this paper, and sketch how Theorem~\ref{thm:quantitativeoppenheim0} follows from it.
Sections~3-6 are devoted to proving Theorem~\ref{thm:mainthm}. In Section~3, we define height functions on $\bR^3\setminus\set{0}$ involving distance to the locus $Q_0=0$ and prove contraction inequalities for such functions. In Section~4, we define modified Margulis functions on the space of lattices and prove subharmonic estimates for such Margulis functions using the contraction inequalities from Section~3. In Section~5, we prove avoidance estimates to control the contribution of the points not satisfying the subharmonic estimates. In Section~6, we complete the proof of Theorem~\ref{thm:mainthm} by assembling the ingredients we developed in Sections~3-5. In Section~7 we deduce Theorem~\ref{thm:quantitativeoppenheim0} and Theorem~\ref{thm:quantitativeoppenheim} from Theorem~\ref{thm:mainthm}.

\vspace{5mm}
\tb{Acknowledgments}. This paper is part of my Ph.D. thesis conducted at ETH Z\"{u}rich under the guidance of Prof. Manfred Einsiedler. I am deeply grateful to him for his insightful discussions, particularly on isotropic vectors of quadratic forms, as well as for carefully reviewing an earlier version of this manuscript and providing invaluable feedback. I also thank Hee Oh for helpful discussions that significantly improved the exposition of this paper.

\section{Preliminaries and Overview of the proofs}

We shall use the standard notation $A\ll B$ or $A=O(B)$ to mean that $A\leq CB$ for some constant $C>0$. In this paper, all the implied constants are absolute, unless mentioned explicitly.

\subsection{Metrics and norms}
Let $\bd_{G}(\cdot,\cdot)$ be a right invariant Riemannian metric on $G$. Then this metric induces a metric $\bd_{X}(\cdot,\cdot)$ on $X$. We denote by $B^G(r)$ the open $r$-ball around $\operatorname{id}$ with respect to the metric $\bd_{G}$.

For $v=(v_1,v_2,v_3)\in\bR^3$ and $\gamma=(\gamma_{ij})_{1\leq i,j\leq 3}\in\operatorname{Mat}_{3,3}(\bR)$, we use the supremum norms $\|v\|=\displaystyle\max_{1\leq i\leq 3}|v_i|$ and $\|\gamma\|=\displaystyle\max_{1\leq i,j\leq 3}\|\gamma_{ij}\|$. 

By re-scaling the metric $\bd_G$ if necessary, we may assume that $\bd_G$ satisfies
\eqlabel{opnorm}{\|g-\textrm{id}\|_{\textrm{op}}\leq r}
for any sufficiently small $r>0$ and $g\in B^G(r)$, where $\|\cdot\|_{\textrm{op}}$ stands for the operator norm of $\operatorname{Mat}_{3,3}(\bR)$ with respect to the supremum norm of $\bR^3$.

\subsection{Diophantine conditions and quasi-null vectors}
For an indefinite ternary quadratic form $Q$, if there are five nonzero integral vectors in $\set{v\in\mathbb{R}^3:Q(v)=0}$ for which no three of these vectors lie on the same plane, then $Q$ must be a rational quadratic form. The following lemma is a quantitative refinement of this observation.

\begin{lem}\label{lem:effecivekernel}
Let $Q$ be an indefinite ternary quadratic form with $\operatorname{det}Q=1$ and $\bm_1,\ldots,\bm_5\in\bZ^3\setminus\set{0}$ be integral vectors for which no three of these vectors lie on the same plane through the origin. Let $0\leq\eps<1$ and $R>10$. Suppose that $|Q(\bm_i)|\leq\eps$ and $\|\bm_i\|<R$ and all $1\leq i\leq 5$. Then there exists a nonzero integral ternary quadratic form $Q'$ satisfying $|Q-\rho Q'|\ll \eps R^{10}$ and $\|Q'\|\leq 10^6R^{14}$, where $\rho=(\operatorname{det}Q')^{-\frac{1}{3}}$.
\end{lem}
\begin{proof}
    Let $\gamma\in \operatorname{Mat}_{3,3}(\bZ)$ be the matrix whose columns are given by $\bm_1,\bm_2,\bm_3$. Note that $\|\gamma\|<R$ and $\|\operatorname{adj}(\gamma)\|<2R^2$, where $\operatorname{adj}(\gamma):=(\operatorname{det}\gamma)\gamma^{-1}$.  Also, $\operatorname{det}\gamma$ is a nonzero integer and $|\operatorname{det}\gamma|\leq 6R^{3}$ since $\bm_1,\bm_2,\bm_3$ are linearly independent integral vectors. Let $Q_1:=Q\circ \gamma$ and write
    \eqlabel{eq:Q1expansion}{\begin{aligned}
    Q_1(v_1,v_2,v_3)&=Q(v_1\bm_1+v_2\bm_2+v_3\bm_3)\\&=\sum_{i=1}^{3}Q(\bm_i)v_i^2+2\sum_{1\leq i<j\leq 3}Q(\bm_i,\bm_j)v_iv_j,
    \end{aligned}}
    where $Q(\bm_i,\bm_j):=\frac{1}{2}\big(Q(\bm_i+\bm_j)-Q(\bm_i)-Q(\bm_j)\big)$. Note that $|Q(\bm_i)|\leq\eps$ for all $1\leq i\leq 3$ and $|Q(\bm_i,\bm_j)|\leq 10\|Q\|R^2$ for all $1\leq i<j\leq 3$. Let $a_1,a_2,a_3,b_1,b_2,b_3\in\bZ$ be the integers given by
    $$\transp{(a_1,a_2,a_3)}=\operatorname{adj}(\gamma)\bm_4, \quad \transp{(b_1,b_2,b_3)}=\operatorname{adj}(\gamma)\bm_5 $$
    Then we may also write
$$(\operatorname{det}\gamma)\bm_4=a_1\bm_1+a_2\bm_2+a_3\bm_3,$$
$$(\operatorname{det}\gamma)\bm_5=b_1\bm_1+b_2\bm_2+b_3\bm_3.$$ Moreover, $\displaystyle\max_{1\leq i\leq 3}\max(|a_i|,|b_i|)\leq 3\|\operatorname{adj}(\gamma)\|\max(\|\bm_4\|,\|\bm_5\|)<6R^3$. In view of \eqref{eq:Q1expansion} it gives that
$$\left|\sum_{1\leq i<j\leq 3}Q(\bm_i,\bm_j)a_ia_j\right|=\frac{1}{2}\left|(\operatorname{det}\gamma)^2Q(\bm_4)-\sum_{i=1}^{3}Q(\bm_i)a_i^2\right|\ll \eps R^6,$$
$$\left|\sum_{1\leq i<j\leq 3}Q(\bm_i,\bm_j)b_ib_j\right|=\frac{1}{2}\left|(\operatorname{det}\gamma)^2Q(\bm_5)-\sum_{i=1}^{3}Q(\bm_i)b_i^2\right|\ll \eps R^6.$$
Roughly speaking, the vector $\big(Q(\bm_2,\bm_3),Q(\bm_3,\bm_1),Q(\bm_1,\bm_2)\big)$ is almost orthogonal to integral vectors $\ba:=(a_3a_1,a_1a_2,a_2a_3)$ and $\bb:=(b_3b_1,b_1b_2,b_2b_3)$. Here $\ba$ and $\bb$ are linearly independent because no three of $\bm_1, \ldots, \bm_5$ lie in the same plane, and $\bm_4$ and $\bm_5$ are themselves linearly independent. Let $$(c_1,c_2,c_3):=\ba\times\bb=\big(a_1b_1(a_2b_3-a_3b_2),a_2b_2(a_3b_1-a_1b_3),a_3b_3(a_1b_2-a_2b_1)\big)$$ and consider an integral ternary quadratic form $$Q_1'(v_1,v_2,v_3)=c_1v_2v_3+c_2v_1v_3+c_3v_1v_2.$$ Note that $Q_1'$ satisfies $\|Q_1'\|\leq 10^4R^{12}$ and 
$\|Q_1-(\operatorname{det}(Q_1'))^{-\frac{1}{3}}Q_1'\|\ll\eps R^6$.
Now we set $Q':=Q_1'\circ \operatorname{adj}(\gamma)$, then it holds that $$\|Q'\|\leq 3\|\operatorname{adj}(\gamma)\|\|Q_1'\|\leq10^6R^{14}$$ and
$$\|Q-(\operatorname{det}(Q'))^{-\frac{1}{3}}Q'\|\leq 3\|\operatorname{adj}(\gamma)\| \|Q_1-(\operatorname{det}(Q_1'))^{-\frac{1}{3}}Q_1'\| \ll \eps R^{10},$$
as desired.
\end{proof}

Lemma \ref{lem:effecivekernel} directly implies the following qualitative statement.

\begin{lem}\label{lem:fourisotropic}
    Let $Q$ be an indefinite irrational ternary quadratic form with $\operatorname{det}Q=1$. There are at most four isotropic rational lines and degenerate rational planes through the origin, respectively. 
\end{lem}
\begin{proof}
    Since for irrational $Q$ the dual quadratic form $Q^*$ is also irrational, by duality it suffices to show the statement for $\Del_Q$-rational lines. Let $Q=Q_0^g$ for some $g\in G$. Suppose for contradiction that there are five isotropic rational lines through the origin. Then we may find $\bm_1,\ldots,\bm_5\in\bZ^3\setminus\set{0}$ for which $Q(\bm_i)=Q_0(g\bm_i)=0$ for $1\leq i\leq 5$ and no pair among these vectors lie on the same line. 
    
    Note that $g\bm_i$'s are on the light cone $\cH:=\set{v\in\mathbb{R}^3:Q_0(v)=0}$, and the intersection of $\cH$ and any plane through the origin is the union of at most two lines through the origin. Thus, no three of $g\bm_1,\ldots,g\bm_5$ lie on the same plane through the origin. Then Lemma~\ref{lem:effecivekernel} with $\eps=0$ contradicts the assumption that $Q$ is irrational.
\end{proof}

\begin{defi}
    For $\eta>0$ and $M>1$ we say that a vector $v\in\bR^3$ is $(\eta,M)$-quasi-null if $|Q_0(v)|<\eta\|v\|^{-50M}$. We denote by $\cH_{\eta,M}$ the set of $(\eta,M)$-quasi-null vectors, i.e. $$\cH_{\eta,M}:=\set{v\in\bR^3: |Q_0(v)|<\eta\|v\|^{-50M}}.$$
\end{defi}

The following lemma asserts that for a quadratic form $Q$ of Diophantine type $M$, there cannot be many $(\eta,M)$-quasi-null vectors in $\Del_Q$.

\begin{lem}\label{lem:sixplanes}
    Let $M>1$. For an indefinite ternary quadratic form $Q$ of Diophantine type $M$ with $\operatorname{det}Q=1$, there exists $0<\eta<1$ such that the following holds. For any $R>10$ the set $\set{v\in \Del_Q\cap \cH_{\eta,M}: R\leq \|v\|<R^2}$ is contained in at most six planes.
\end{lem}
\begin{proof}
    Suppose for contradiction that for any $0<\eta<1$ there exists $R=R(
    \eta
    )>10$ such that the set $$\set{v\in \Del_Q\cap \cH_{\eta,M}: R\leq \|v\|<R^2}$$ is not contained in any union of six planes through the origin. Then we can find five nonzero vectors $\bv_1,\ldots,\bv_5\in \Del_Q\cap\cH_{\eta,M}$ for which no three of them lie on the same plane through the origin and $R\leq\|\bv_i\|<R^2$ for $1\leq i\leq 5$. Let $Q=Q_0^g$ for some $g\in G$. Then there are nonzero integral vectors $\bm_1,\ldots,\bm_5$ such that $\bv_i=g\bm_i$ and for $1\leq i\leq 5$. Since $Q(\bm_i)=Q_0(\bv_i)$ and $\bv_i\in\cH_{\eta,M}$, we have
    $$|Q(\bm_i)|=|Q_0(\bv_i)|\leq \eta\|\bv_i\|^{-50M}\leq \eta R^{-50M}$$
    and $\|\bm_i\|\leq \|g^{-1}\|\|\bv_i\|< \|g^{-1}\|R^2$ for $1\leq i\leq 5$. By Lemma~\ref{lem:effecivekernel} there exists a nonzero integral ternary quadratic form $Q'$ satisfying $$\|Q-\rho Q'\|\ll (\eta R^{-50M})(\|g^{-1}\|R^2)^{10}= \eta \|g^{-1}\|R^{-50M+20}$$ and $\|Q'\|\leq 10^6\|g^{-1}\|^{14}R^{28}$, where $\rho=(\operatorname{det}Q')^{-\frac{1}{3}}$. For any $c>0$ this implies $$\|Q-\rho Q'\|\leq \eta \|g^{-1}\|R^{-28M}\leq c\|Q'\|^{-M},$$ provided $\eta$ is sufficiently small. However, this contradicts that $Q$ is of Diophantine type $M$. This completes the proof.
\end{proof}

\begin{lem}\label{lem:qunatitativenondegeneracy}
    Let $Q$ be an indefinite ternary quadratic form with $\operatorname{det}Q=1$. Let $L\subset\mathbb{R}^3$ be a plane, and let $\bm_1,\bm_2,\bm_3\in\mathbb{Z}^3\in L$ be integral vectors such that $\bm_i$ and $\bm_j$ are not on the same line for all $1\leq i\neq j\leq 3$. For $0\leq \epsilon<1$ and $R>1$ if $|Q(\bm_i)|\leq \epsilon$ and $R<\|\bm_i\|<R^2$ for all $1\leq i\leq 3$, then $\epsilon\gg R^{-64}$ holds, where the implied constant depends on $Q$. 
\end{lem}
\begin{proof}
    Let us take a $\mathbb{Z}$-basis $\set{\bn_1,\bn_2}$ of $L\cap\mathbb{Z}^3$. As $\|\bm_i\|\leq R^2$ for $1\leq i\leq 3$, we may choose $\set{\bn_1,\bn_2}$ so that $\|\bn_1\|,\|\bn_2\|\leq R^2$. For $1\leq i\leq 3$ we may write $\bm_i=a_{i1}\bn_1+a_{i2}\bn_2$ for some $a_{i1},a_{i2}\in\mathbb{Z}$ with $|a_{i1}|,|a_{i2}|\ll R^2$. Then we have $Q(\bm_i)=a_{i1}^2Q(\bn_1)+2a_{i1}a_{i2}Q(\bn_1,\bn_2)+a_{i2}^2Q(\bn_2)$ for $1\leq i\leq 3$, i.e.
    $$\left(\begin{matrix}
        a_{11}^2 & a_{11}a_{12} & a_{12}^2 \\ a_{21}^2 & a_{21}a_{22} & a_{22}^2\\ a_{31}^2 & a_{31}a_{32} & a_{32}^2
    \end{matrix}\right) \left(\begin{matrix}Q(\bn_1) \\ 2Q(\bn_1,\bn_2)\\ Q(\bn_2)
    \end{matrix}\right)=\left(\begin{matrix}Q(\bm_1) \\ Q(\bm_2)\\ Q(\bm_3)
    \end{matrix}\right).$$
    Let $\gamma=\left(\begin{matrix}
        a_{11}^2 & a_{11}a_{12} & a_{12}^2 \\ a_{21}^2 & a_{21}a_{22} & a_{22}^2\\ a_{31}^2 & a_{31}a_{32} & a_{32}^2
    \end{matrix}\right)$. Note that $\|\gamma\|\ll R^4$, $\|\operatorname{adj}(\gamma)\|\ll R^8$, and moreover, by a projective version of the Van der Monde matrix we have a formula
    $$\operatorname{det}(\gamma)=\prod_{1\leq i\neq j\leq 3}(a_{i1}a_{j2}-a_{i2}a_{j1})\in\mathbb{Z}\setminus\set{0}.$$
     Here, $\operatorname{det}\gamma\neq 0$ is by the assumption that $\bm_i$ and $\bm_j$ are not on the same line for all $1\leq i< j\leq 3$. It follows that $|Q(\bn_1)|,|Q(\bn_1,\bn_2)|,|Q(\bn_2)|$ are bounded by $$\ll \|\gamma^{-1}\|\max_{1\leq i\leq 3}|Q(\bm_i)|\leq \|\operatorname{adj}(\gamma)\|\max_{1\leq i\leq 3}|Q(\bm_i)|\ll R^8\epsilon.$$

    Let $\bn_3=\bn_1\times\bn_2$ and let $\gamma'\in\operatorname{Mat}_{3,3}(\mathbb{Z})$ be the matrix whose columns are $\bn_1,\bn_2,\bn_3$. For $Q_1=Q\circ(\gamma')^{-1}$ we have $Q_1(\be_i,\be_j)=Q(\bn_i,\bn_j)$ for all $1\leq i,j\leq 3$. Note that $\|\gamma'\|\ll R^4$, $\|\operatorname{adj}(\gamma')\|\ll R^8$, hence $\|(\gamma')^{-1}\|\leq |\operatorname{det}\gamma'|^{-1}\|\operatorname{adj}(\gamma')\|\ll R^8$. This implies that $|Q_1(\be_i,\be_j)|\leq \|Q\circ (\gamma')^{-1}\|\ll \|(\gamma')^{-1}\|^2\ll R^{16}$ for all $1\leq i,j\leq 3$. We also know that $|Q_1(\be_i,\be_j)|=|Q(\bn_i,\bn_j)|\ll R^8\epsilon$ for $1\leq i,j\leq 2$, so $ |\operatorname{det}Q_1|\ll R^{40}\epsilon$. Therefore we get
    $$1=|\operatorname{det}Q|=|\operatorname{det}Q_1||\operatorname{det}\gamma'|^2\ll (R^{40}\epsilon)(R^{12})^{2}=R^{64}\epsilon.$$    
\end{proof}

Combining this with Lemma \ref{lem:sixplanes}, we obtain the following.

\begin{lem}\label{lem:twelvelines}
    Let $M>1$. For an indefinite ternary quadratic form $Q$ of Diophantine type $M$ with $\operatorname{det}Q=1$, there exists $0<\eta<1$ such that the following holds. For any $R>10$ the set $\set{v\in \Del_Q\cap \cH_{\eta,M}: R\leq \|v\|<R^2}$ is contained in at most $12$ lines.
\end{lem}
\begin{proof}
    According to Lemma~\ref{lem:sixplanes} the set $\set{v\in \Del_Q\cap \cH_{\eta,M}: R\leq \|v\|<R^2}$ is contained in at most six planes.
\end{proof}

\subsection{Margulis functions excluding quasi-null vectors}
To improve the exponent $\lambda$ in the moment of the Margulis function $\int_{-1}^{1}\alpha(a_tu_r\Del_Q)^{\lambda}dr$ beyond $1$, one needs to exclude the contribution of quasi-null vectors in $\Del_Q$. We introduce the following modification of $\alpha$ excluding isotropic or quasi-null vectors. We define $\widehat{\alpha}_1,\widehat{\alpha}_2,\widehat{\alpha}:H \times X\to(0,\infty)$ by $$\widehat{\alpha}_1(g;\Delta):=\sup\set{\|gv\|^{-1}:v\in\Delta\setminus\cH},\quad\widehat{\alpha}_2(g;\Delta):=\widehat{\alpha}_1(g^*;\Delta^*),$$
$$\widehat{\alpha}(g;\Del)=\max\set{\widehat{\alpha}_1(g;\Delta),\widehat{\alpha}_2(g;\Delta)}.$$

For $\eta>0$ and $M>1$ we also define $\widehat{\alpha}_{1,\eta,M},\widehat{\alpha}_{2,\eta,M},\widehat{\alpha}_{\eta,M},\widehat{\alpha}_{\eta,M}':H \times X\to (0,\infty)$ by
\eqlabel{eq:alphahat12def}{\widehat{\alpha}_{1,\eta,M}(g;\Delta):=\sup\set{\|gv\|^{-1}:v\in \Delta\setminus \cH_{\eta,M}},\quad \widehat{\alpha}_{2,\eta,M}(g;\Delta):=\widehat{\alpha}_{1,\eta,M}(g^*;\Delta^*),}
\eqlabel{eq:alphahatdef}{
    \widehat{\alpha}_{\eta,M}(g;\Delta):=\max\set{\widehat{\alpha}_{1,\eta,M}(g;\Delta),\widehat{\alpha}_{2,\eta,M}(g;\Delta)},}
\eqlabel{eq:alphahat'def}{\widehat{\alpha}'_{\eta,M}(g;\Delta):=\max\set{\widehat{\alpha}_{\eta,M}(g;\Delta),\alpha(g\Delta)^{0.9}.}}
By definitions, for any $\eta$, $M$, and $\Del$ we have
$$\widehat{\alpha}(g;\Delta)\leq \alpha(g\Del)\leq \widehat{\alpha}(g;\Delta)+\!\!\!\sup_{v\in \Del\cap\cH\setminus\set{0}}\!\!\!\|gv\|^{-1}+\!\!\!\sup_{v\in \Del^{*}\cap\cH\setminus\set{0}}\!\!\!\|g^*v\|^{-1},$$
$$\widehat{\alpha}_{\eta,M}(g;\Delta)\leq \widehat{\alpha}(g;\Delta)\leq \alpha(g\Del)\leq \widehat{\alpha}_{\eta,M}(g;\Delta)+\!\!\!\!\!\!\sup_{v\in \Del\cap\cH_{\eta,M}\setminus\set{0}}\!\!\!\!\!\!\|g^*v\|^{-1}+\!\!\!\!\!\!\sup_{v\in \Del^{*}\cap\cH_{\eta,M}\setminus\set{0}}\!\!\!\!\!\!\|gv\|^{-1},$$
$$ \widehat{\alpha}_{\eta,M}'(g;\Delta)\leq \widehat{\alpha}_{\eta,M}(g;\Delta)+\!\!\!\!\!\!\sup_{\-v\in \Del\cap\cH_{\eta,M}\setminus\set{0}}\!\!\!\!\!\!\|gv\|^{-0.9}+\!\!\!\!\!\!\sup_{\-v\in \Del^{*}\cap\cH_{\eta,M}\setminus\set{0}}\!\!\!\!\!\!\|g^*v\|^{-0.9}.$$

Let $J:=\left(\begin{matrix}
     & & 1\\ & -1 & \\ 1 & &
\end{matrix}\right)\in H$. Then we have $J=J^{-1}=\transp{ J}=J^*$. We observe that
$$Ja_tJ^{-1}=a_t^*,\quad Ju_rJ^{-1}=u_r^*,\quad J\transp{u_r}J^{-1}=u_r^{-1}=(\transp{u_r})^*$$
for any $t,r\in\mathbb{R}$, hence $JgJ^{-1}=g^*$ holds for any $g\in H$. We also observe that $J$ preserves $\|\cdot\|$ and $Q_0(\cdot)$, since $Jv=(v_3,-v_2,v_1)$ for $v=(v_1,v_2,v_3)\in\mathbb{R}^3$. It follows that $J$ preserves $\cH$ and $\cH_{\eta,M}$ for any $\eta>0$ and $M>1$.

We readily see from these properties that for any $g\in H$ and $\Delta\in X$
\eq{\begin{aligned}
    \widehat{\alpha}_1(g;\Delta)&=\sup\set{\|gv\|^{-1}:v\in\Delta\setminus\cH}\\&=\sup\set{\|g^*Jv\|^{-1}:v\in\Delta\setminus\cH}\\&=\sup\set{\|g^*v\|^{-1}:v\in J\Delta\setminus\cH}\\&=\widehat{\alpha}_2(g;J\Delta^*),
\end{aligned}}
hence
$$\widehat{\alpha}(g;\Delta)=\max\set{\widehat{\alpha}_1(g;\Delta),\widehat{\alpha}_2(g;\Delta)}=\max\set{\widehat{\alpha}_1(g;\Delta),\widehat{\alpha}_1(g;J\Delta^*)}=\widehat{\alpha}(g;J\Delta^*).$$
Similarly, we also have
$$\widehat{\alpha}_{1,\eta,M}(g;\Delta)=\widehat{\alpha}_{2,\eta,M}(g;J\Delta^*), \quad \widehat{\alpha}_{\eta,M}(g;\Delta)=\widehat{\alpha}_{\eta,M}(g;J\Delta^*), \quad \widehat{\alpha}_{\eta,M}'(g;\Delta)=\widehat{\alpha}_{\eta,M}'(g;J\Delta^*)$$
for any $\eta>0$, $M>1$, $g\in H$ and $\Delta\in X$.



\subsection{Outline of the proof of Theorem \ref{thm:mainthm}}\label{subsec:outline}
Let $0<\del\leq 0.01$. In Section~3 we define a function $\phi_\del:\bR^3\setminus\set{0}\to(0,\infty]$ by $\phi_\del(v)=\kappa(v)^{-2\del}\|v\|^{-1-\del}$, where roughly speaking $\kappa:\bR^3\setminus\set{0}\to[0,1]$ measures the distance to the surface $\cH$. Then $\phi_\del$ satisfies $\phi_\del(v)\geq \|v\|^{-1-\del}$ for all $v\in \bR^3\setminus\set{0}$, and the following contraction inequality holds (Proposition~\ref{eq:linearcontractionhypothesis}):
\eqlabel{eq:linearcontractionhypothesisintro}{\int_{-1}^{1}\phi_\del(a_su_rv)dr\leq 80\del^{-1}e^{-\del s}\phi_\del(v)}
for any $s\geq 1$ and $v\in\bR^3\setminus\set{0}$.

In Section~4, using the function $\phi_\del$, for $\eta>0$ and $M>1$ we define a modified height function $\widetilde{\alpha}_{\del,\eta,M}:H\times X\to[1,\infty]$ satisfying $$\widetilde{\alpha}_{\del,\eta,M}(g;\Del)\geq \max\set{\widehat{\alpha}_{\eta,M}(g;\Del)^{1+\del},\alpha(g\Del)^{1-3\del}}$$ for any $(g,\Del)\in H\times X$. We then establish the following subharmonic estimate (Proposition~\ref{prop:contractionhypothesis}), using \eqref{eq:linearcontractionhypothesisintro}. If $\del>0$ is a sufficiently small constant depending on $M$, then for any large $s$ it holds that
\eqlabel{eq:contractionhypothesisintro}{\int_{-1}^{1}\widetilde{\alpha}_{\del,\eta,M}(a_su_rg;\Del)dr\leq e^{-\frac{\del}{2} s}\widetilde{\alpha}_{\del,\eta,M}(g;\Del)+e^{9s},}
unless $(g,\Del)$ lies within a certain set $\cE_{s,\eta,M}\subset H\times X$, where the fibers over $H$ have small volume.

To derive \eqref{eq:contractionhypothesisintro} from \eqref{eq:linearcontractionhypothesisintro} we follow closely the strategy of \cite{EMM98}. However, here we further need to control the effect of the factor $\kappa(gv)^{-2\del}$, which is large when the vector $gv\in g\Del$ is close to the surface $\cH$. The estimate for such contribution is provided by a supremum-version of the contraction inequality for $\phi_\del$ (Proposition~\ref{eq:refinedcontractionhypothesis}), where the supremum is taken over the vectors very close to $\cH$. The main ingredient of the proof of Proposition~\ref{eq:refinedcontractionhypothesis} is a quantitative version of the following simple geometric observation: for any plane in $\bR^3$, the intersection of the plane and $\cH$ is the union of at most two lines in $\bR^3$.

We remark that \eqref{eq:contractionhypothesisintro} holds only for $(g;\Del)\notin \cE_{s,\eta,M}$, whereas the classical contraction hypotheses hold for every $\Del$. Thus, in Section~5 we control the amount of time that the orbit $(a_tu_r;\Del)$ stays in the set $\cE_{s,\eta,M}$ when $t$ is sufficiently larger than $s$. Namely, in Proposition \ref{prop:avoidanceMargulis'} we show that
\eqlabel{eq:avoidanceMargulisintro}{\int_{-1}^{1}\widehat{\alpha}_{\eta,M}'(a_tu_r;\Delta)^{1+\del}\mathds{1}_{\cE_{s,\eta,M}}(a_tu_r;\Delta)dr\ll e^{-10s}}
for any $s\geq 1$ and $t\geq 4DMs$, where $D$ is an absolute constant. The proof of the estimate \eqref{eq:avoidanceMargulisintro} relies on an effective avoidance principle of Sanchez and Seong~\cite[Theorem 2]{SS22} (see also \cite[Proposition 4.6]{LMW22}).

In Section~6, we combine the subharmonic estimate \eqref{eq:contractionhypothesisintro} from Section~3 and the avoidance estimate \eqref{eq:avoidanceMargulisintro} from Section 4, hence completing the proof of Theorem~\ref{thm:mainthm}.

\subsection{Sketch of the proof of Theorem~\ref{thm:quantitativeoppenheim0} from Theorem~\ref{thm:mainthm}}\label{sec:momentfromcounting}
Let $f$ be a bounded function defined on $\bR^3\setminus\set{0}$ vanishing outside a compact set. For $\eta>0$, $M>1$, and $\Del\in X$ we also denote by $Y_{\eta,M}(\Del)$ the set of vectors $v\in \Del$ satisfying
$$\|Q_0(v)\|\geq \eta\|v\|^{-50M},\qquad \|Q_0(v\times w)\|\geq \eta\|v\times w\|^{-50M}$$
for any $w\in \Del$ with $v\times w\neq0$.

For $\Del\in X$ and $g\in G$ let
$$\widetilde{f}(\Del):=\sum_{v\in\Del}f(v),\qquad\widehat{f}(g;\Del):=\sum_{v\in gY(\Del)}f(v).$$
We recall the Siegel mean value formula:
$$\int_X \widetilde{f}(\Del)dm_X(\Del)=\int_{\bR^3}f(v)dv$$
for any bounded and compactly supported function $f$.

For $\eta>0$ and $M>1$ we also define $\widehat{f}_{\eta,M}$ by
$$\widehat{f}_{\eta,M}(g;\Del):=\sum_{v\in gY_{\eta,M}(\Del)}f(v).$$

\begin{lem}[Lipschitz principle]\label{lem:Lipschitzprinciple}
    Let $\eta>0$ and $M>1$. For any $f\in C_0^{\infty}(\bR^3\setminus\set{0})$ there exists a constant $c=c(f)$ such that for all $\Del\in X$ and $g\in G$
    $$\widetilde{f}(\Del)\leq c\alpha(\Del), \quad \widehat{f}(g;\Del)\leq c\widehat{\alpha}(g;\Del), \quad \widehat{f}_{\eta,M}(g;\Del)\leq c\widehat{\alpha}_{\eta,M}(g;\Del).$$
\end{lem}
\begin{proof}
    The first inequality $\widetilde{f}(\Del)\leq c\alpha(\Del)$ is the Lipschitz principle of Schmidt, and the other two inequalities are direct modifications; see \cite[Lemma 2]{Sch68}.
\end{proof}

In \S7.2 we will prove the following proposition, using Theorem~\ref{thm:mainthm}.

\begin{prop}[Equidistribution of Siegel transform]\label{prop:unboundedRatner}
    Let $Q$ be an indefinite quadratic form with $\operatorname{det}Q=1$, which is not EWA. For any bounded Riemann integrable function $f$ compactly supported on $\bR^3\setminus\set{0}$ and $\nu\in C(K)$ we have
    $$\lim_{t\to\infty}\int_{K}\widehat{f}(a_tk;\Del_Q)\nu(k)dk=\int_{\bR^3} f(v)dv\int_K \nu(k)dk.$$
\end{prop}

 Then Theorem~\ref{thm:quantitativeoppenheim0} is deduced from Proposition \ref{prop:unboundedRatner} by an argument identical to that used in \cite[\S3.4, \S3.5]{EMM98} to deduce \cite[Theorem 2.1]{EMM98} from \cite[Theorem 2.3]{EMM98}. Roughly speaking, for the region $\cW\subset \bR^3\setminus\set{0}$ defined by
 $$\cW:=\set{v\in\mathbb{R}^3: a<Q_0(v)<b, \;\tfrac{1}{2}\leq \|v\|\leq 1, \;v_1>0,\;\tfrac{1}{2}v_1\leq |v_2|\leq v_1},$$
 using the fact that $K$ preserves both $\|\cdot\|$ and $Q_0(\cdot)$ one can calculate that
 $$ce^t\lim_{t\to\infty}\int_{K}\widehat{\mathds{1}_{\cW}}(a_tk;\Del_Q)dk\approx\begin{cases}
        1 \quad \textrm{ if } \tfrac{e^t}{2}\leq\|v\|\leq e^t\;,a<Q_0(v)<b,\\ 0 \quad\textrm{ otherwise}
    \end{cases} $$
    for some constant $c>0$, where $\mathds{1}_{\cW}$ is the characteristic function on $\cW$. This implies that
 $$\#\set{v\in \bZ^3: \frac{1}{2}e^{t}\leq\|v\|\leq e^t \textrm{ and }a<Q(v)<b}\approx ce^t\lim_{t\to\infty}\int_{K}\widehat{\mathds{1}_{\cW}}(a_tk;\Del_Q)dk.$$ Based on this relation, Theorem~\ref{thm:quantitativeoppenheim0} follows from Proposition~\ref{prop:unboundedRatner} with $f=\mathds{1}_{\cW}$ and $\nu=\mathds{1}_K$.

\section{Linear actions on $\bR^3$}

\subsection{Known estimates}
In this subsection, we review several estimates from \cite{EMM98}, along with certain variants derived from them. We first recall the following contraction inequality from \cite[Lemma 5.1]{EMM98}.
\begin{lem}[Contraction for $\|\cdot\|^{-\lambda}$]\label{eq:triviallinearcontraction0}
    Let $\frac{1}{2}\leq \lambda\leq 1$. For any $t>0$ and $w\in\bR^3\setminus\set{0}$ we have
    $$\int_{-1}^{1}\|a_tu_rw\|^{-\lambda}dr\leq 100e^{-\frac{(1-\lambda) t}{3}}\|w\|^{-\lambda}.$$   
\end{lem}

Utilising the arguments in \cite[Lemma 5.1]{EMM98} (or the proof of Proposition~\ref{eq:linearcontractionhypothesis} in this section alternatively) one can also obtain:
\begin{lem}[Bounded expansion for $\|\cdot\|^{-1-\del}$]\label{eq:triviallinearcontraction}
    Let $0<\del<\frac{1}{2}$. For any $t>0$ and $w\in\bR^3\setminus\set{0}$ we have
    $$\int_{-1}^{1}\|a_tu_rw\|^{-1-\del}dr\leq 40\del^{-1}e^{\del t}\|w\|^{-1-\del}.$$
\end{lem}

We note that the analogous inequalities to those in Lemma~\ref{eq:triviallinearcontraction0} and Lemma~\ref{eq:triviallinearcontraction} also hold when $a_tu_rw$ is replaced by $a_t^*u_r^*w$.

As in \cite[\S 5]{EMM98}, Lemma~\ref{eq:triviallinearcontraction0} and Lemma~\ref{eq:triviallinearcontraction} indeed imply the following estimates, respectively, without any assumption on $\Delta\in X$.

\begin{lem}[Subharmonic estimate for $\alpha^{\lambda}$]\label{eq:triviallinearcontraction1}
    Let $\frac{1}{2}\leq \lambda< 1$. For any $t>0$ and $\Del\in X$ we have
    $$\int_{-1}^{1}\alpha(a_tu_r\Del)^{\lambda}dr\leq 100e^{-\frac{(1-\lambda) t}{3}}\alpha(\Del)^{\lambda}+e^{4t}.$$
\end{lem}

\begin{lem}[Superharmonic estimate for $\alpha^{1+\del}$]\label{eq:trivialcontractionhypothesis}
    Let $0<\del<\frac{1}{2}$. For any $\Del\in X$ and $t>0$ we have
    $$\int_{-1}^{1}\alpha(a_tu_r\Del)^{1+\del}dr\leq 400\del^{-1}e^{\del t}\alpha(\Del)^{1+\del}+e^{4t}.$$
\end{lem}

Note that for any $t>0$, $r\in\mathbb{R}$, and $w\in\mathbb{R}^3$ the following log-Lipschitz property holds:
\eqlabel{eq:Lipschitz0}{(3e^t)^{-1}\|w\|\leq\|a_tu_rw\|\leq 3e^t\|w\|.} We also note expansion bounds for $\widehat{\alpha}_{\eta,M}$ and $\widehat{\alpha}_{\eta,M}'$, which are deduced from Lemma~\ref{eq:triviallinearcontraction}.

\begin{lem}[An expansion bound for $\widehat{\alpha}_{\eta,M}^{1+\del}$]\label{eq:trivialcontractionhypothesishat}
    Let $0<\eta<1$, $M>1$, and $0<\del<0.01$. For any $g\in H$ and $\Del\in X$ and $t>0$ we have
    $$\int_{-1}^{1}\widehat{\alpha}_{\eta,M}(a_tu_rg;\Del)^{1+\del}dr\leq 80\del^{-1}e^{\del t}\widehat{\alpha}_{\eta,M}(g;\Del)^{1+\del}+80\delta^{-1}\alpha(g\Delta)^{0.9}+\tfrac{1}{2}e^{6t}.$$
\end{lem}
\begin{proof}
    Let $v\in\Delta\setminus\cH_{\eta,M}$ be a vector such that $\widehat{\alpha}_{1,\eta,M}(g;\Delta)=\|gv\|^{-1}$. Suppose that there exists $w\in \Delta\setminus\cH_{\eta,M}$ with $v\wedge w\neq 0$ and $\|gw\|\leq (3e^t)^{2}\|gv\|$. Then we have
    $$\alpha_2(g\Delta)^{-1}\leq \|gv\wedge gw\|\leq \|gv\|\|gw\|\leq (3e^t)^2\widehat{\alpha}_{1,\eta,M}(g;\Delta)^{-2},$$ hence $\widehat{\alpha}_{1,\eta,M}(g;\Delta)\leq 3e^t\alpha_2(g\Delta)^{\frac{1}{2}}$. It follows that $$\widehat{\alpha}_{1,\eta,M}(a_tu_rg;\Delta)\leq (3e^t)\widehat{\alpha}_{1,\eta,M}(g;\Delta)\leq (3e^t)^2\alpha_2(g\Delta)^{\frac{1}{2}}$$ for any $r\in[-1,1]$.
    
    If there is no such $w\in \Delta\setminus\cH_{\eta,M}$, i.e. $\|gw\|>(3e^t)^2\|gv\|$ for all $w\in \Delta\setminus\cH_{\eta,M}$ with $v\wedge w\neq 0$, then we see that $\|a_tu_rgv\|\leq \|a_tu_rgw\|$ for all $w\in \Delta\setminus\cH_{\eta,M}$, hence $\widehat{\alpha}_{1,\eta,M}(a_tu_rg;\Delta)=\|a_tu_rv\|^{-1}$.
    
    Combining these two cases, we deduce that
    \eq{\begin{aligned}
        \widehat{\alpha}_{1,\eta,M}(a_tu_rg;\Delta)^{1+\delta}&\leq \|a_tu_rgv\|^{-1-\delta}+\left(9e^{2t}\alpha_2(g\Delta)^{\frac{1}{2}}\right)^{1+\delta}\\&\leq \|a_tu_rgv\|^{-1-\delta}+20\delta^{-1}\alpha(g\Delta)^{0.9}+\tfrac{1}{8}e^{6t},
    \end{aligned}}
    using Young's inequality for products, i.e. $xy\leq \frac{1}{p}x^p+\frac{1}{q}y^q$ for $p,q> 1$ with $\frac{1}{p}+\frac{1}{q}=1$ and $x,y\geq 0$. Applying Lemma~\ref{eq:triviallinearcontraction} we have
    \eq{\begin{aligned}
        \int_{-1}^{1}\widehat{\alpha}_{1,\eta,M}(a_tu_rg;\Delta)^{1+\delta}dr&\leq \int_{-1}^{1}\|a_tu_rgv\|^{-1-\delta}dr+40\delta^{-1}\alpha(g\Delta)^{0.9}+\tfrac{1}{4}e^{6t}\\&\leq 40\delta^{-1}e^{\delta t}\|gv\|^{-1-\delta}+40\delta^{-1}\alpha(g\Delta)^{0.9}+\tfrac{1}{4}e^{6t}.\\&=40\delta^{-1}e^{\delta t}\widehat{\alpha}_{1,\eta,M}(g;\Delta)^{1+\delta}+40\delta^{-1}\alpha(g\Delta)^{0.9}+\tfrac{1}{4}e^{6t}.
    \end{aligned}}
    By a similar argument for $\widehat{\alpha}_{2,\eta,M}(g;\Delta)$ we also have
    $$\int_{-1}^{1}\widehat{\alpha}_{2,\eta,M}(a_tu_rg;\Delta)^{1+\delta}dr\leq  40\delta^{-1}e^{\delta t}\widehat{\alpha}_{2,\eta,M}(g;\Delta)^{1+\delta}+40\delta^{-1}\alpha(g\Delta)^{0.9}+\tfrac{1}{4}e^{6t},$$
    hence the desired estimate follows, as $\displaystyle\sup_{i=1,2}\widehat{\alpha}_{i,\eta,M}= \widehat{\alpha}_{\eta,M}\leq \widehat{\alpha}_{1,\eta,M}+\widehat{\alpha}_{2,\eta,M}$.
\end{proof}

\begin{lem}[Superharmonic estimate for $\widehat{\alpha}_{\eta,M}'^{1+\del}$]\label{eq:trivialcontractionhypothesishat'}
    Let $0<\eta<1$, $M>1$, and $0<\del<0.01$. For any $g\in H$ and $\Del\in X$ and $t>10$ we have
    $$\int_{-1}^{1}\widehat{\alpha}_{\eta,M}'(a_tu_rg;\Del)^{1+\del}dr\leq 400\delta^{-1}e^{\delta t}\widehat{\alpha}'_{\eta,M}(g;\Delta)^{1+\delta}+e^{6t}.$$
\end{lem}
\begin{proof}
    Using Lemma~\ref{eq:trivialcontractionhypothesishat} and Lemma~\ref{eq:triviallinearcontraction1} we deduce 
\eq{\begin{aligned}
    \int_{-1}^{1}\widehat{\alpha}_{\eta,M}'(a_tu_rg;\Del)^{1+\del}dr&\leq \int_{-1}^{1}\widehat{\alpha}_{\eta,M}(a_tu_rg;\Del)^{1+\del}dr+\int_{-1}^{1}\alpha(a_tu_rg\Del)^{0.9(1+\delta)}dr\\&\leq \big(80\del^{-1}e^{\del t}\widehat{\alpha}_{\eta,M}(g;\Del)^{1+\del}+80\delta^{-1}\alpha(g\Delta)^{0.9}+\tfrac{1}{2}e^{6t}\big)\\&\qquad\qquad\qquad\qquad\qquad+\big(100e^{-0.01t}\alpha(g\Delta)^{0.9(1+\delta)}+e^{4t}\big)
    \\&\leq 400\delta^{-1}e^{\delta t}\widehat{\alpha}'_{\eta,M}(g;\Delta)^{1+\delta}+e^{6t}.
\end{aligned}}
\end{proof}

\subsection{Auxiliary functions}
For $w=(w_1,w_2,w_3)\in \bR^3\setminus\set{0}$ we define three quantifiers. For $w$ with $w_2=w_3=0$ we set $\rho(w)=\infty$, $\kappa_0(w)=\infty$, and $\kappa(w)=1$. Assuming $(w_2,w_3)\neq (0,0)$ we define
\eqlabel{eq:rhodef}{\rho(w):=-\frac{w_2}{w_3}\in[-\infty,\infty],}
\eqlabel{eq:kappa0def}{\kappa_0(w):=\frac{Q_0(w)}{w_3^2}\in[-\infty,\infty],}
\eqlabel{eq:kappadef}{\kappa(w):=\begin{cases}
    |\kappa_0(w)| & \text{ if } |\kappa_0(w)|<1 \text{ and } |\rho(w)|< 2,\\
    1 & \text{ otherwise}.
\end{cases}}

We first observe that $\rho(a_tu_rw)=e^t(\rho(w)-r)$ and  $\kappa_0(a_tu_rw)=e^{2t}\kappa_0(w)$ for any $t,r\in\bR$. By a straightforward matrix calculation, we also observe that
\eqlabel{eq:matrixexpansion}{a_tu_rw=w_3\left(\begin{matrix}
    e^t(w_1+rw_2+\tfrac{1}{2}r^2w_3)\\ w_2+rw_3 \\ e^{-t}w_3
\end{matrix}\right)=w_3\left(\begin{matrix}
    \frac{1}{2}e^t\left\{(r-\rho(w))^2-\kappa_0(w)\right\}\\ r-\rho(w) \\ e^{-t}
\end{matrix}\right)}
for any $t,r\in\bR$, and $w\in\bR^2\times(\bR\setminus\set{0})$.
\begin{lem}[Growth of $\kappa$]\label{lem:kappaexpansion}
    For any $t\geq 1$, $|r|\leq 1$, and $w\in\bR^3\setminus\set{0}$ we have \eq{\kappa(a_tu_rw)\geq\begin{cases}
    1 & \text{ if } \kappa(w)\geq e^{-2t},\\
    e^{2t}\kappa(w) & \text{ if } \kappa(w)< e^{-2t}.
\end{cases}}
Furthermore, if $\kappa(a_tu_rw)<1$ then $\kappa(a_tu_rw)=e^{2t}\kappa(w)$.
\end{lem}
\begin{proof}
    Since $a_t$ and $u_r$ are in $H$, they stabilize $Q_0$.
    
    \textbf{Case 1}. We suppose $\kappa(w)\geq e^{-2t}$.\\In this case we have either $|\rho(w)|\geq 2$ or $|\kappa_0(w)|\geq e^{-2t}$. If $|\rho(w)|\geq 2$, then $|\rho(a_tu_rw)|\geq e^t|r-\rho(w)|=e^t(|\rho(w)|-|r|)> 2$, hence $\kappa(a_tu_rw)=1$. If $|\kappa_0(w)|\geq e^{-2t}$, then $|\kappa_0(a_tu_rw)|=e^{2t}|\kappa_0(w)|\geq1$, hence $\kappa(a_tu_rw)=1$.\\
    
    \textbf{Case 2}. We suppose $\kappa(w)< e^{-2t}$.\\
    In this case we have $\kappa(w)=|\kappa_0(w)|$, hence
\eqlabel{eq:kappaexpansion}{\kappa(a_tu_rw)\geq|\kappa_0(a_tu_rw)|=e^{2t}|\kappa_0(w)|=e^{2t}\kappa(w).}

From \textbf{Case 1} we also see that $\kappa(a_tu_rw)<1$ implies $\kappa(w)<1$, hence $\kappa(a_tu_rw)=|\kappa_0(a_tu_rw)|$ and $\kappa(w)=|\kappa_0(w)|$. It follows that
\eq{\kappa(a_tu_rw)=|\kappa_0(a_tu_rw)|=e^{2t}|\kappa_0(w)|=e^{2t}\kappa(w).}

\end{proof}


\subsection{Contraction for linear actions on $\bR^3$}
In this subsection, we construct a modified local height function $\phi_\del$ for $0<\del<0.01$ and establish the corresponding contraction inequality. As noted in Remark~\ref{rem:doublezero}, the contraction inequality for $\|\cdot\|^{-\lambda}$ in Lemma~\ref{eq:triviallinearcontraction0} is no longer valid when the exponent $\lambda>1$. In particular, we observe that
$$\lim_{t\to\infty}\int_{-1}^{1}\|a_tu_rv\|^{-\lambda}dr=\infty $$
for any $\lambda>1$, if the quadratic equation $w_1+rw_2+\tfrac{1}{2}r^2w_3$ in \eqref{eq:matrixexpansion} has a double root and satisfies $w_2+rw_3=0$ for some $r\in[-1,1]$, i.e. $\kappa_0(v)=0$ and $|\rho(w)|\leq 1$. However, if this double-zero scenario is excluded, the range of permissible exponents $\lambda$ can be refined. This observation motivates the construction of the following modified local height function $\phi_\del$ on $\mathbb{R}^3$, which satisfies $\phi_\del(w)\geq\|w\|^{-1-\delta}$ and exhibits the desired contraction property. 

For $0<\del<0.01$ let us define $\phi_\del:\bR^3\setminus\set{0}\to (0,\infty]$ by
\eqlabel{eq:phidef}{\phi_\del(w)=\kappa(w)^{-2\del}\|w\|^{-1-\del}.}
For any $0\leq \del<0.1$, $t\geq 1$, $r\in[-1,1]$, and $w\in \bR^3\setminus\set{0}$, the following log-Lipschitz property holds by \eqref{eq:Lipschitz0} and Lemma~\ref{lem:kappaexpansion}:
\eqlabel{eq:Lipschitz}{\phi_\del(a_tu_rw)\leq (3e^t)^{1+\del}\phi_\del(w).}
Moreover, if $\kappa(a_tu_rw)<1$ then
\eqlabel{eq:Lipschitz'}{\phi_\del(a_tu_rw)\geq (3e^t)^{-(1+5\del)}\phi_\del(w).}

This subsection aims to show the following contraction inequality for $\phi_\del$.

\begin{prop}[Contraction for $\phi_\del$]\label{eq:linearcontractionhypothesis}
    Let $0<\del\leq 0.01$. Then we have
    $$\int_{-1}^{1}\phi_\del(a_tu_rw)dr\leq 80\del^{-1}e^{-\del t}\phi_\del(w)$$
    for any $t\geq 1$ and $w\in\bR^3\setminus\set{0}$.
\end{prop}

\begin{proof}
Let $w=(w_1,w_2,w_3)$. If $w_2=w_3=0$ then $a_tu_rw=e^tw$ for any $r\in\bR$, hence the statement is trivial. If $w_2\neq0$ and $w_3=0$, then $\kappa(a_tu_rw)=\kappa(w)=1$ for any $r\in\bR$, and it is straightforward to see that
\eq{\begin{aligned}
    \int_{-1}^{1}\phi_\del(a_tu_rw)dr&=\int_{-1}^{1}\|a_tu_rw\|^{-1-\del}dr\\&=\int_{-1}^{1}\max\set{e^t|w_1+w_2r|,|w_2|}^{-1-\del}dr \\&\leq 10\del^{-1}e^{-t}\|w\|^{-1-\del}=10\del^{-1}e^{-t}\phi_\del(w)
\end{aligned}}
for any $t\geq 1$. We now assume $w_3\neq 0$ and let 
$$f_w(r):=(r-\rho(w))^2-\kappa_0(w)=\frac{2}{w_3}\big(w_1+w_2r+\tfrac{1}{2}w_3r^2\big).$$\\
\textbf{Case 1}. $\kappa(w)\leq e^{-t}$.
    
    By Lemma \ref{lem:kappaexpansion} we have $\kappa(a_tu_rw)\geq e^{t}\kappa(w)$ for any $t\geq 1$ and $|r|<1$. In combination with the expansion bound for $\|\cdot\|^{-1-\del}$ in Lemma \ref{eq:triviallinearcontraction} it follows that
    \eq{\begin{aligned}
        \int_{-1}^{1}\phi_\del(a_tu_rw)dr&=\int_{-1}^{1}\kappa(a_tu_rw)^{-2\del}\|a_tu_rw\|^{-1-\del}dr\\&\leq e^{-2\del t}\kappa(w)^{-2\del}\int_{-1}^{1}\|a_tu_rw\|^{-1-\del}dr\\&\leq e^{-2\del t}\kappa(w)^{-2\del}(40\del^{-1}e^{\del t}\|w\|^{-1-\del}) \\&= 40\del^{-1}e^{-\del t}\kappa(w)^{-2\del}\|w\|^{-1-\del}=40\del^{-1}e^{-\del t}\phi_\del(w).
    \end{aligned}}\\
    
    \textbf{Case 2}. $|\rho(w)|\geq 2$ and hence $\kappa(w)=1$.
    
    In this case $\kappa(a_tu_rw)=1$ for any $|r|\leq 1$ by Lemma \ref{lem:kappaexpansion}. Since $f_w'(r)=2(r-\rho(w))$, we have $|f_w'(r)|=2|r-\rho(w)|\geq \frac{1}{2}|\rho(w)|$ for any $r\in [-\frac{3}{2},\frac{3}{2}]$. In particular, $f_w'$ does not change the sign on $[-\frac{3}{2},\frac{3}{2}]$, and there is at most one zero of $f_w$ on $[-\frac{3}{2},\frac{3}{2}]$.

    \textbf{Subcase 2-1}. There is one zero of $f_w$ on $[-\frac{3}{2},\frac{3}{2}]$.

    Let $r_0\in[-\frac{3}{2},\frac{3}{2}]$ be the zero of $f_w$. Then $f_w(r_0)=0$ and $|\rho(2)|\geq 2$ imply that
    $$|w_1|=\left|r_0w_2+\frac{r_0^2}{2}w_3\right|\leq |r_0|\big(|w_2|+\frac{|w_3|}{2}\big)\leq 2|w_2|,$$
    hence $\|w\|=\max(|w_1|,|w_2|,|w_3|)\leq 2|w_2|$. We also observe that
    $$|f_w(r)|=|(r-\rho(w))^2-(r_0-\rho(w))^2|=|r-r_0||r+r_0-2\rho(w)|\geq \tfrac{1}{2}|\rho(w)||r-r_0|$$
    for any $r\in [-\frac{3}{2},\frac{3}{2}]$. It follows that
    \eq{\begin{aligned}
    \int_{-1}^{1}&\phi_\del(a_tu_rw)dr=\int_{-1}^{1}\|a_tu_rw\|^{-1-\del}dr\\&\leq\int_{e^{-t}\leq |r-r_0|<3}\left|\tfrac{1}{2}w_3e^tf_w(r)\right|^{-1-\del}dr+\int_{\substack{|r-r_0|<e^{-t},\\ |r|<1}}\left|w_3(r-\rho(w))\right|^{-1-\del}dr.
    \end{aligned}}
    The first integral is bounded by
    \eq{\begin{aligned}
        \int_{e^{-t}\leq |r-r_0|<3}\left|\tfrac{1}{2}w_3e^tf_w(r)\right|^{-1-\del}dr&\leq 2^{1+\del}|w_3e^t\rho(w)|^{-1-\del}\int_{e^{-t}\leq |r-r_0|<3}|r-r_0|^{-1-\del}dr\\&\leq 2^{2+\del}\del^{-1}e^{\del t}|w_3e^t\rho(w)|^{-1-\del}\leq 2^{2+\del}\del^{-1}e^{-t}|w_2|^{-1-\del}.
    \end{aligned}}
    The second integral is bounded by
    \eq{\begin{aligned}
        \int_{\substack{|r-r_0|<e^{-t},\\ |r|<1}}\left|\tfrac{1}{2}w_3(r-\rho(w))\right|^{-1-\del}dr&\leq 8^{1+\del}\int_{|r-r_0|<e^{-t}}|w_3\rho(w)|^{-1-\del}dr\\&\leq 10e^{-t}|w_3\rho(w)|^{-1-\del}=10e^{-t}|w_2|^{-1-\del}.
    \end{aligned}}
    since $|r-\rho(w)|\geq \frac{1}{4}|\rho(w)|$ for any $r\in[-1,1]$. Hence, for any $t>0$ we have
    $$\int_{-1}^{1}\phi_\del(a_tu_rw)dr\leq (10+2^{2+\del}\del^{-1})e^{-t}|w_2|^{-1-\del}\leq 10\del^{-1}e^{-t}\|w\|^{-1-\del}=10\del^{-1}e^{-t}\phi_\del(w).$$
    
    \textbf{Subcase 2-2}. There is no zero of $f_w$ on $[-\frac{3}{2},\frac{3}{2}]$.

    Recall that $f_w'$ does not change the sign on $[-\frac{3}{2},\frac{3}{2}]$ and $|f_w'(r)|=2|r-\rho(w)|\geq \frac{1}{2}|\rho(w)|$ for any $r\in [-\frac{3}{2},\frac{3}{2}]$. It follows that $|f_w(r)|\geq \frac{1}{4}|\rho(w)|$ for any $r\in[-1,1]$. On the other hand, we also have
    \eq{\begin{aligned}
        |w_3f_w(r)|&=2|w_1+rw_2+\tfrac{1}{2}r^2w_3|\\&\geq 2\left(|w_1|-|w_2|-\frac{|w_2|}{2|\rho(w)|}\right)\geq 2|w_1|-3|w_2|
    \end{aligned}}
    for any $r\in [-1,1]$. It follows that
    $$|w_3f_w(r)|\geq \max\set{\frac{1}{4}|w_2|,2|w_1|-3|w_2|}\geq \frac{1}{8}\max(|w_1|,|w_2|)\geq \frac{1}{8}\|w\|.$$
     We thus obtain an estimate that
    \eq{\begin{aligned}
        \int_{-1}^{1}\phi_\del(a_tu_rw)dr&=\int_{-1}^{1}\|a_tu_rw\|^{-1-\del}dr\\&\leq 2^{1+\del}\int_{-1}^{1}\left|w_3e^tf_w(r)\right|^{-1-\del}dr\\&< 32^{1+\del}e^{-(1+\del)t}\|w\|^{-1}\leq 100e^{-t}\phi_\del(w).
    \end{aligned}}\\    
    \textbf{Case 3}. $\kappa(w)\geq e^{-t}$ and $|\rho(w)|<2$.
    
    In this case $\kappa(a_tu_rw)=1$ by Lemma \ref{lem:kappaexpansion}, and $|w_2|\leq 2|w_3|$.

    \textbf{Subcase 3-1}. $\left|\kappa_0(w)\right|\geq 10$.
    
    We have $|2w_1w_3|\geq |Q_0(w)|-w_2^2\geq 10w_3^2-4w_3^2=6w_3^2$, hence $|w_1|\geq 3|w_3|$ and $\|w\|=|w_1|$. We also see $w_2^2\leq |\rho(w)|^2w_3^2\leq 4w_3^2\leq \frac{4}{3}|w_1w_3|$, hence $$|Q_0(w)|\geq |2w_1w_3|-w_2^2\geq \tfrac{2}{3}|w_1w_3|.$$ Since $(r-\rho(w))^2\leq 9$, we get
    $$|f_w(r)|\geq \left|\kappa_0(w)\right|-|r-\rho(w)|^2\geq\tfrac{1}{10}|\kappa_0(w)|= \frac{|Q_0(w)|}{10w_3^2}\geq\frac{|w_1|}{15|w_3|}$$
    for any $r\in[-1,1]$. It follows that
    \eq{\begin{aligned}
        \int_{-1}^{1}\phi_\del(a_tu_rw)dr&=\int_{-1}^{1}\|a_tu_rw\|^{-1-\del}dr\\&\leq 2^{1+\del}\int_{-1}^{1}|w_3e^tf_w(r)|^{-1-\del}dr\\&< 60^{1+\del}e^{-(1+\del)t}|w_1|^{-1-\del}\leq 100e^{-t}\|w\|^{-1-\del}=100e^{-t}\phi_\del(w).
    \end{aligned}}

    \textbf{Subcase 3-2}. $e^{-t}\leq -\kappa_0(w)<10$.
    
    We have $|2w_1w_3|\leq |Q_0(w)|+w_2^2\leq 14w_3^2$, hence $|w_1|\leq 7|w_3|$ and $\|w\|\leq 7|w_3|$. Since $|f_w(r)|\geq (r-\rho(w))^2+e^{-t}$, we see that
    \eq{\begin{aligned}
        \int_{-1}^{1}\phi_\del(a_tu_rw)dr&=\int_{-1}^{1}\|a_tu_rw\|^{-1-\del}dr\\&\leq 2^{1+\del}\int_{-1}^{1}|w_3e^tf_w(r)|^{-1-\del}dr\\&\leq 2^{1+\del}|w_3e^t|^{-1-\del}\int_{-1}^{1}\frac{dr}{\{(r-\rho(w))^2+e^{-t}\}^{1+\del}}\\&\leq 2^{1+\del}|w_3e^t|^{-1-\del}\int_{-3}^{3}\frac{dr}{(r^2+e^{-t})^{1+\del}}\\&<2^{1+\del}|w_3e^t|^{-1-\del}(\pi e^{(\frac{1}{2}+\del) t})\\&\leq 14^{1+\del}\pi e^{-\frac{t}{2}}\|w\|^{-1-\del}\leq 100e^{-\frac{t}{2}}\phi_\del(w).
    \end{aligned}}

    \textbf{Subcase 3-3}. $e^{-t}\leq \kappa_0(w)<10$.
    
    As in \textbf{Subcase 3-2} we have $\|w\|\leq 7|w_3|$. Note that $f_w(r)$ has two real roots $r_1=\rho(w)-\sqrt{\kappa_0(w)}$ and $r_2=\rho(w)+\sqrt{\kappa_0(w)}$. Then we may write $|f_w(r)|=|r-r_1||r-r_2|$. Let $I_1=[r_1-e^{-\frac{3}{2}t},r_1+e^{-\frac{3}{2}t}]$ and $I_2=[r_2-e^{-\frac{3}{2}t},r_2+e^{-\frac{3}{2}t}]$. Then we have
    \eq{\begin{aligned}    \int_{-6}^{\rho(w)}\phi_\del(a_tu_rw)dr&=\int_{-6}^{\rho(w)}\|a_tu_rw\|^{-1-\del}dr\\&\leq \int_{[-6,\rho(w)]\setminus I_1}\left|\tfrac{1}{2}w_3e^tf_w(r)\right|^{-1-\del}dr+\int_{I_1}\left|\tfrac{1}{2}w_3e^{-t}\right|^{-1-\del}dr.
    \end{aligned}}
    The first integral in the last line is bounded by
    \eq{\begin{aligned}
    &\leq 2^{1+\del}|w_3e^t|^{-1-\del}\int_{[-6,\rho(w)]\setminus I_1}|r-r_1|^{-1-\del}|r-r_2|^{-1-\del}dr\\
        &\leq 2^{1+\del}|w_3e^t|^{-1-\del}\left|\kappa_0(w)\right|^{-\frac{1+\del}{2}}\int_{[-6,\rho(w)]\setminus I_1}|r-r_1|^{-1-\del}dr\\&\leq 2^{1+\del}|w_3|^{-1-\del}e^{-\frac{1+\del}{2}t}\int_{e^{-\frac{3}{2}t}<|r|<10}|r|^{-1-\del}dr\\&< 4^{1+\del}\del^{-1}|w_3|^{-1-\del}e^{-(\frac{1}{2}-\del)t}.
    \end{aligned}}
    The second integral is bounded as follows:
    $$2^{1+\del}\int_{I_1}|w_3e^{-t}|^{-1-\del}dr<  4^{1+\del}|w_3|^{-1-\del}e^{-(\frac{1}{2}-\del)t}.$$
    Thus, we get the following estimate:
    $$ \int_{-6}^{\rho(w)}\phi_\del(a_tu_rw)dr\leq 30\del^{-1}e^{-\frac{t}{3}}\|w\|^{-1-\del}\leq 40\del^{-1}e^{-\frac{t}{3}}\phi_\del(w).$$
    Similarly, we also have
    $$\int_{\rho(w)}^{6}\phi_\del(a_tu_rw)dr\leq 40\del^{-1}e^{-\frac{t}{3}}\phi_\del(w),$$
    hence
    \eq{\begin{aligned}
        \int_{-1}^{1}\phi_\del(a_tu_rw)dr&\leq \int_{-6}^{\rho(w)}\phi_\del(a_tu_rw)dr+\int_{\rho(w)}^{6}\phi_\del(a_tu_rw)dr\\&\leq 80\del^{-1}e^{-\frac{t}{3}}\phi_\del(w).
    \end{aligned}}
\end{proof}

We define $\kappa^\star:\mathbb{R}^3\setminus\set{0}\to [0,1]$ and $\phi_\del^\star:\bR^3\setminus\set{0}\to (0,\infty]$ by $\kappa^\star(w)=\kappa(Jw)$ and
\eqlabel{eq:phi*def}{\phi_\del^\star(w)=\phi_{\del}(Jw)=\kappa^\star(w)^{-2\del}\|w\|^{-1-\del}.}
The contraction inequality for $\phi_\del^\star$ follows from Proposition~\ref{eq:linearcontractionhypothesis}.

\begin{prop}[Contraction for $\phi_\del^\star$]\label{eq:linearcontractionhypothesis*}
    Let $0<\del\leq 0.01$. Then we have
    $$\int_{-1}^{1}\phi_\del^\star(a_t^*u_r^*w)dr\leq 80\del^{-1}e^{-\del t}\phi_\del^\star(w)$$
    for any $t\geq 1$ and $w\in\bR^3\setminus\set{0}$.
\end{prop}
\begin{proof}
    In combination with \eqref{eq:phi*def}, Proposition~\ref{eq:linearcontractionhypothesis} implies that
    \eq{\begin{aligned}
        \int_{-1}^{1}\phi_\del^\star(a_t^*u_r^*w)dr&=\int_{-1}^{1}\phi_\del(a_tu_rJw)dr\\&\leq 80\del^{-1}e^{-\del t}\phi_\del(Jw)=80\del^{-1}e^{-\del t}\phi_\del^\star(w).
    \end{aligned}}
\end{proof}

\section{Margulis inequality}\label{sec:MargulisInequality}

\subsection{Construction of modified Margulis function $\widetilde{\alpha}_{\del,\eta,M}$}
In this subsection, we modify the original Margulis function $\alpha$ so that a modified height function $\widetilde{\alpha}_{\del,\eta,M}$ involves an additional factor related to the distance from $\cH$. The construction of the modified height function $\widetilde{\alpha}_{\del,\eta,M}$ relies on the local height function $\phi_\del$ we studied in the previous section. We will also state the subharmonic estimate for the modified height function.

For $0<\del\leq 0.01$, $0<\eta<1$, and $M>1$ define $\widehat{\phi}_{1,\del,\eta,M},\widehat{\phi}_{2,\del,\eta,M}:H\times (\bR^3\setminus\set{0})\to (0,\infty)$ by
\eqlabel{eq:phihatdef}{\widehat{\phi}_{1,\del,\eta,M}(g;v):=\begin{cases}
    \phi_\del(gv) & \text{ if } v\notin \cH_{\eta,M},\\
    \|gv\|^{-(1-3\del)} & \text{ if } v\in \cH_{\eta,M},
\end{cases}}
\eqlabel{eq:phihat*def}{\widehat{\phi}_{2,\del,\eta,M}(g;v):=\begin{cases}
    \phi_\del^\star(g^*v) & \text{ if } v\notin \cH_{\eta,M},\\
    \|g^*v\|^{-(1-3\del)} & \text{ if } v\in \cH_{\eta,M},
\end{cases}}
where $g\in G$ and $v\in \bR^3\setminus\set{0}$. For any $g\in H$ and $v\in \bR^3\setminus\set{0}$ we have \eqlabel{eq:pihatdualrelation}{\widehat{\phi}_{1,\del,\eta,M}(g;v)=\widehat{\phi}_{2,\del,\eta,M}(g;Jv),} since $\|v\|=\|Jv\|$ and $Q_0(v)=Q_0(Jv)$. 

From now on, we will write $\kappa_1=\kappa$ and $\kappa_2=\kappa^\star$ for notational convenience.

\begin{prop}[Contraction for $\widehat{\phi}_{i,\del,\eta,M}$]\label{eq:linearcontractionhypthesishat}
    For any $i=1,2$, $0<\del\leq 0.01$, $s\geq 1$, $r\in[-1,1]$, $g\in H$, and $v\in \bR^3\setminus\set{0}$ we have
    \eqlabel{eq:phihatcontraction}{\int_{-1}^{1}\widehat{\phi}_{i,\del,\eta,M}(a_su_rg;v)dr\leq 80\del^{-1}e^{-\del s}\widehat{\phi}_{i,\del,\eta,M}(g;v).}
    Moreover, we have
\eqlabel{eq:Lipschitzhat}{\widehat{\phi}_{i,\del,\eta,M}(a_su_rg;v)\leq (3e^s)^{1+\del}\widehat{\phi}_{i,\del,\eta,M}(g;v).}
We also have
\eqlabel{eq:Lipschitz'hat}{\widehat{\phi}_{i,\del,\eta,M}(a_su_rg;v)\geq (3e^s)^{-(1+5\del)}\widehat{\phi}_{i,\del,\eta,M}(g;v)\quad \textrm{ if }\kappa_i(a_su_rgv)<1.}
    
\end{prop}
\begin{proof}
    The log-Lipschitz properties \eqref{eq:Lipschitzhat} and \eqref{eq:Lipschitz'hat} follow from \eqref{eq:Lipschitz0}, \eqref{eq:Lipschitz}, and \eqref{eq:Lipschitz'}. The contraction inequality \eqref{eq:phihatcontraction} follows from the contraction inequality for $\|\cdot\|^{-(1-3\del)}$ in Lemma \ref{eq:triviallinearcontraction0} for $v\in \cH_{\eta,M}$ and the contraction inequality for $\phi_{\del}$ in Proposition \ref{eq:linearcontractionhypothesis} for $v\notin \cH_{\eta,M}$.
\end{proof}

Let us denote $\Lambda^{\circ}:=\Lambda\cap B(1)$ for any discrete subset $\Lambda$ of $\bR^3$. For $0<\delta\leq 0.01$, $0<\eta<1$, and $M>1$ we define $\widetilde{\alpha}_{1,\del,\eta,M}:H\times X\to [1,\infty]$ and $\widetilde{\alpha}_{2,\del,\eta,M}:H\times X\to [1,\infty]$ by
\eqlabel{eq:alphadel1def}{\widetilde{\alpha}_{1,\del,\eta,M}(g;\Del):=\begin{cases}
    \sup\set{\widehat{\phi}_{1,\del,\eta,M}(g;v): gv\in (g\Del)^{\circ}\setminus\set{0}} & \text{ if } (g\Del)^{\circ}\neq\set{0},\\
    1 & \text{ otherwise},
\end{cases}}
\eqlabel{eq:alphadel2def}{\widetilde{\alpha}_{2,\del,\eta,M}(g;\Del):=\begin{cases}
    \sup\set{\widehat{\phi}_{2,\del,\eta,M}(g;v): g
    ^*v\in (g^*\Del^*)^{\circ}\setminus\set{0}} & \text{ if } (g^*\Del^*)^{\circ}\neq\set{0},\\
    1 & \text{ otherwise},
\end{cases}}
where $g^*\Del^*$ is the dual lattice of $g\Del$. 
Now we define the modified height function $\widetilde{\alpha}_{\del,\eta,M}:H\times X\to[1,\infty]$ by
\eqlabel{eq:alphadeldef}{\widetilde{\alpha}_{\del,\eta,M}(g;\Del):=\max\set{\widetilde{\alpha}_{1,\del,\eta,M}(g;\Del),\widetilde{\alpha}_{2,\del,\eta,M}(g;\Del)}}
for any $g\in H$ and $\Del\in X$. Note that we have
\eqlabel{eq:alphatildedualrelation}{\widetilde{\alpha}_{1,\del,\eta,M}(g;\Del)=\widetilde{\alpha}_{2,\del,\eta,M}(g;J\Del^*),\quad \widetilde{\alpha}_{\del,\eta,M}(g;\Del)=\widetilde{\alpha}_{\del,\eta,M}(g;J\Del^*).}
Recall that $\phi_\del(v)\geq \|v\|^{-1-\del}$ for all $v\in\bR^3\setminus\set{0}$. By Minkowski's first theorem either $(g\Del)^{\circ}\neq\set{0}$ or $\alpha_1(g\Del)=1$ holds, hence we clearly have $$\widetilde{\alpha}_{i,\del,\eta,M}(g;\Del)\geq \max\big(\widehat{\alpha}_{i,\eta,M}(g;\Del)^{1+\del},\alpha_i(g\Del)^{1-3\del}\big)$$ for any $i=1,2$ and $(g,\Del)\in H\times X$. This shows that  $$\widetilde{\alpha}_{\del,\eta,M}(g;\Del)\geq \max\big(\widehat{\alpha}_{\eta,M}(g;\Del)^{1+\del},\alpha(g\Del)^{1-3\del}\big)$$
for any $g\in H$ and $\Del\in X$.

We fix a certain absolute constant $D>1000$, which will be determined later in \S5 (see Proposition~\ref{prop:avoidanceMargulis'}).

For $s\geq 1$, $0<\eta\leq1$, and $M>1$, we define $\eps_{s,\eta,M}:H\times X\to (0,1)$ by
\eqlabel{eq:epsdef}{\eps_{s,\eta,M}(g;\Del):=\begin{cases}
    \eta (3e^{s})^{-60M} & \text{ if } \widehat{\alpha}_{\eta,M}(g;\Del)\leq 10^4e^{4s},\\
    \eta\widehat{\alpha}_{\eta,M}(g;\Del)^{-100DM^2} & \text{ if } \widehat{\alpha}_{\eta,M}(g;\Del)> 10^4e^{4s}.
\end{cases}}
For $i=1,2$, $s\geq 1$ and $0<\eps<1$ we denote
\eqlabel{eq:Xidef}{\Xi_i(s,\eps):=\set{v\in\bR^3: 1\leq\|v\|\leq 3e^s,\; \kappa_i(v)<\eps}.}
We define $\cE_{s,\eta,M}=\cE_{1,s,\eta,M}\cup\cE_{2,s,\eta,M}$ by
\eqlabel{eq:cEdef}{\cE_{1,s,\eta,M}:=\set{(g;\Del)\in H\times X: g(\Del\setminus \cH_{\eta,M})\cap\Xi_1(s,\eps_{s,\eta,M}(g;\Del))\neq\emptyset},}
\eqlabel{eq:cEdef*}{\cE_{2,s,\eta,M}:=\set{(g;\Del)\in H\times X: g^*(\Del^*\setminus \cH_{\eta,M})\cap\Xi_2(s,\eps_{s,\eta,M}(g;\Del))\neq\emptyset}.}
We note that $\Xi_1(s,\eps)=J\Xi_2(s,\eps)$, hence $\cE_{1,s,\eta,M}=\set{(g,J\Delta^*):(g;\Delta)\in\cE_{2,s,\eta,M}}$ holds.

We see that for any $s\geq1$, $0<\eta\leq 1$, $M>1$, and  $\Del\in X$
\eqlabel{eq:identitycontained}{(\operatorname{id},\Del)\notin \cE_{s,\eta,M},}
as $\Xi_1(s,\eps_{s,\eta,M}(\operatorname{id},\Del))$ and $\Xi_2(s,\eps_{s,\eta,M}(\operatorname{id},\Del))$ are contained in $\cH_{\eta,M}$ by definition. Indeed, for $i=1,2$ if $v\in \Xi_i(s,\eps_{s,\eta,M}(\operatorname{id},\Del))$ then $$|Q_0(v)|\leq \|v\|^2\kappa_i(v)<9e^{2s}\epsilon_{s,\eta,M}(\operatorname{id},\Delta)\leq 9\eta e^{2s}(3e^s)^{-60M}<\eta \|v\|^{-50M},$$
hence $v\in \cH_{\eta,M}$.

We shall record the following lemma:

\begin{lem}\label{lem:Q0Diophantineexpanding}
    Let $\Del\in X$, $s\geq 1$, $t\geq 4DMs$, and $r\in[-1,1]$. If $\widehat{\alpha}_{\eta,M}(a_tu_r;\Del)>10^4e^{\frac{t}{DM}}$ then $(a_tu_r,\Del)\notin\cE_{s,\eta,M}$.
\end{lem}
\begin{proof}
    Assume for contradiction that there exist $s\geq 1$, $t\geq 4DMs$, and $r\in[-1,1]$ such that $\widehat{\alpha}_{\eta,M}(a_tu_r;\Del)>10^4e^{\frac{t}{DM}}$ and $(a_tu_r,\Del)\in\cE_{s,\eta,M}$. Note that $$\eps_{s,\eta,M}(a_tu_r;\Del)=\eta\widehat{\alpha}_{\eta,M}(a_tu_r;\Del)^{-100DM^2}$$ since $\widehat{\alpha}_{\eta,M}(a_tu_r;\Del)>10^4e^{4s}$. 
    
    Without loss of generality, we may assume that $(a_tu_r,\Delta)\in\cE_{1,s,\eta,M}$. Then there exists $v\in \Del\setminus\cH_{\eta,M}$ such that 
    $$a_tu_rv\in \Xi_1(s,\eta\widehat{\alpha}_{\eta,M}(a_tu_r;\Del)^{-100DM^2})\subseteq \Xi_1(s,\eta 10^{-400DM^2}e^{-100Mt}).$$ 
    This implies that $1\leq \|a_tu_rv\|\leq 3e^s$ and $\kappa(a_tu_rv)< 10^{-400DM^2} \eta e^{-100Mt}$, hence using \eqref{eq:Lipschitz0}, \eqref{eq:kappa0def}, and \eqref{eq:kappadef} we have $\|v\|\leq 3e^t\|a_tu_rv\|<10e^{t+s}$ and \eqlabel{eq:Q0lowerbound}{|Q_0(v)|=|Q_0(a_tu_rv)|\leq\kappa(a_tu_rv)\|a_tu_rv\|^{2}<10^{-100M}\eta e^{-100Mt+2s}.}
    On the other hand, since $v\notin \cH_{\eta,M}$ it holds that
    $$|Q_0(v)|>\eta \|v\|^{-50M}> \eta (10e^{t+s})^{-50M}> 10^{-100M}\eta e^{-100Mt+2s},$$
    but this contradicts \eqref{eq:Q0lowerbound}.
\end{proof}

The rest of this section will be devoted to proving the following subharmonic estimate for the modified height function $\widetilde{\alpha}_{\del,\eta,M}$.

\begin{prop}[Subharmonic estimate for $\widetilde{\alpha}_{\del,\eta,M}$]\label{prop:contractionhypothesis}
    Let $s\geq 1$, $0<\eta\leq 1$, $M>1$, and $0<\del\leq\frac{1}{400DM(M+7)}$. Then for any $g\in H$ and $\Del\in X$ with $(g;\Del)\notin \cE_{s,\eta,M}$ we have
    \eqlabel{eq:alphacontraction}{\int_{-1}^{1}\widetilde{\alpha}_{\del,\eta,M}(a_su_rg;\Del)dr\leq \del^{-10}\eta^{-4\del}e^{-\del s}\widetilde{\alpha}_{\del,\eta,M}(g;\Del)+e^{9s}.}
\end{prop}

\subsection{Intersection of surface and plane}
We begin with a simple observation. Let $L$ be a plane in $\bR^3$. Then the intersection of $L$ and the light cone $\cH=\set{v\in\bR^3:Q_0(v)=0}$ is the union of at most two lines in $\bR^3$. The contents of this subsection are based on a quantitative version of this observation.

 For $s\geq 1$, $\Del\in X$, and a $\Del$-rational plane $L\subset\bR^3$, denote by $\Omega_0(\Del,L,s)$ the set of nonzero vectors $v\in\Del\cap L$ such that $|\kappa_0(v)|<e^{-3s}$. The following lemma describes the distribution of $\Omega_0(\Del,L,s)$ on the plane $L$.
\begin{lem}\label{lem:intersection}
     For $s\geq 1$, $\Del\in X$, and a $\Del$-rational plane $L\subset\bR^3$, there exists a finite set $\cR_L\subset [-1,1]$ with $|\cR_L|\leq 2$ satisfying $\operatorname{dist}(\cR_L,\rho(b))\leq5e^{-s}$ for any $b\in \Omega_0(\Del,L,s)$.
\end{lem}
\begin{proof}
    We may assume that there are two distinct vectors $v,w\in\Omega_0(\Del,L,s)$ with $|\rho(v)-\rho(w)|>5e^{-s}$. We shall show that $\cR_L=\set{\rho(v),\rho(w)}$ satisfies $\operatorname{dist}(\cR_L,\rho(b))\leq5e^{-s}$ for any $b\in\Omega_0(\Del,L,s)$. 
    
    Let $v=(v_1,v_2,v_3)\in\bR^3$ and $w=(w_1,w_2,w_3)\in\bR^3$. Then any vector in $\Del\cap L$ can be written in a form of $mv+nw$ with $(m,n)\in\bQ^2\setminus\set{0}$. For any $mv+nw\in\Omega_0(\Del,L,s)$ we may expand $\kappa_0(mv+nw)$ as follows:
    \eq{\begin{aligned}
        |\kappa_0(mv+nw)|&=\left|2\left(\frac{mv_1+nw_1}{mv_3+nw_3}\right)-\left(\frac{mv_2+nw_2}{mv_3+nw_3}\right)^2\right|\\
        &=\frac{|mnv_3w_3(\rho(v)-\rho(w))^2-(mv_3+nw_3)(mv_3\kappa_0(v)+nw_3\kappa_0(w))|}{(mv_3+nw_3)^2}.
    \end{aligned}}
    For simplicity we denote $\theta_{m,n}=\frac{mnv_3w_3}{(|mv_3|+|nw_3|)^2}$. Since $|\kappa_0(v)|,|\kappa_0(w)|<e^{-3s}$,
    \eq{\begin{aligned}
    |\kappa_0(mv+nw)|&\geq \left|\theta_{m,n}(\rho(v)-\rho(w))^2-\frac{(mv_3+nw_3)(mv_3\kappa_0(v)+nw_3\kappa_0(w))}{(|mv_3|+|nw_3|)^2}\right|\\&\geq |\theta_{m,n}|(\rho(v)-\rho(w))^2-\left|\frac{(mv_3+nw_3)(mv_3\kappa_0(v)+nw_3\kappa_0(w))}{(|mv_3|+|nw_3|)^2}\right|\\&\geq |\theta_{m,n}|(\rho(v)-\rho(w))^2-e^{-3s}.
    \end{aligned}}
    Thus, $mv+nw\in \Omega_0(\Del,L,s)$ implies $|\theta_{m,n}|<\frac{e^{-s}}{10}$. This in turn gives that either $\left|\frac{nw_3}{mv_3}\right|<\frac{e^{-s}}{4}$ or $\left|\frac{mv_3}{nw_3}\right|<\frac{e^{-s}}{4}$.

    Let us first consider the case $\left|\frac{nw_3}{mv_3}\right|<\frac{e^{-s}}{4}$. In this case, we have
    \eq{\begin{aligned}
        |\rho(mv+nw)-\rho(v)|&=\left|
        \frac{mv_2+nw_2}{mv_3+nw_3}-\frac{v_2}{v_3}\right|\\&=\left|\frac{nw_3}{mv_3+nw_3}\right||\rho(v)-\rho(w)|<\frac{e^{-s}}{4-e^{-s}}\cdot 2<e^{-s}.
    \end{aligned}}
    If $\left|\frac{mv_3}{nw_3}\right|<\frac{e^{-s}}{4}$, then we similarly get $|\rho(mv+nw)-\rho(w)|<e^{-s}$. This completes the proof.
\end{proof}

Let us denote by $\Omega(\Del,L,s)$ the set of nonzero vectors $v\in\Del\cap L$ such that $\kappa(v)<e^{-3s}$. The set $\Omega(\Del,L,s)$ is clearly a subset of $\Omega_0(\Del,L,s)$. Now we show the following supremum-version of the contraction inequality for $\phi_\del$ using Lemma \ref{lem:intersection}, where the supremum is taken over a subset of $\Omega(\Del,L,s)$.

\begin{prop}[Supremum-version of contraction for $\phi_\del$]\label{eq:refinedcontractionhypothesis}
    Let $0\leq \del\leq 0.01$ and let $s,\Del,L$ be as in Lemma \ref{lem:intersection}. For any $\cT\subseteq \Om(\Del,L,s)$ we have
    $$\int_{-1}^{1}\sup_{v\in \cT}\phi_\del(a_su_rv)dr\leq 200 e^{-\del s}\sup_{v\in \cT}\phi_\del(v).$$
\end{prop}
\begin{proof}
    Let $\cR_L$ be as in Lemma~\ref{lem:intersection}. We may assume $|\cR_L|=2$ and let $\cR_L=\set{r_1,r_2}$. Then $\min\set{|r_1-\rho(v)|,|r_2-\rho(v)|}\leq 5e^{-s}$ for any $v\in\Omega(\Del,L,s)$.
    
    By Lemma~\ref{lem:kappaexpansion} we have $\kappa(a_su_rv)\geq e^{2s}\kappa(v)$ for all $v\in \Omega(\Del,L,s)$. It follows that
    \eqlabel{eq:kappatrivialbound}{\int_{-1}^{1}\sup_{v\in \cT}\phi_\del(a_su_rv)dr\leq e^{-2\del s}\int_{-1}^{1}\sup_{v\in \cT}\kappa(v)^{-2\del}\|a_su_rv\|^{-1-\del}dr.}
    Let us write $v=(v_1,v_2,v_3)$. Note that $\|v\|\leq 3|v_3|$ for any $v\in \Omega(\Del,L,s)$, since $|\kappa(v)|<e^{-3s}$.
    If $\min\set{|r-r_1|,|r-r_2|}\geq 10e^{-s}$, then
    \eq{\begin{aligned}
        |r-\rho(v)|&\geq \min\set{|r-r_1|-|r_1-\rho(v)|,|r-r_2|-|r_2-\rho(v)|}\\&\geq \tfrac{1}{2}\min\set{|r-r_1|,|r-r_2|}\geq 5e^{-s}
    \end{aligned}}for any $v\in \Omega(\Del,L,s)$. Furthermore, we have
    $$\|a_su_rv\|\geq \tfrac{1}{2}e^s|v_3|((r-\rho(v))^2-\kappa(v))\geq \tfrac{1}{4}e^s|v_3|(r-\rho(v))^2,$$
    hence
    \eqlabel{eq:quadraticbound}{\begin{aligned}
        \|a_su_rv\|^{-1-\del}&\leq 4^{1+\del}|v_3|^{-1-\del}e^{-(1+\del)s}|r-\rho(v)|^{-2(1+\del)}\\&\leq 16^{1+\del}|v_3|^{-1-\del}e^{-(1+\del)s}(|r-r_1|^{-2(1+\del)}+|r-r_2|^{-2(1+\del)})
        \\&\leq 48^{1+\del}\|v\|^{-1-\del}e^{-(1+\del)s}(|r-r_1|^{-2(1+\del)}+|r-r_2|^{-2(1+\del)})
    \end{aligned}}
    for any $v\in \Omega(\Del,L,s)$.

    Let $I_1=[r_1-10e^{-s},r_1+10e^{-s}]$ and $I_2=[r_2-10e^{-s},r_2+10e^{-s}]$. For $r\in [-1,1]\setminus(I_1\cup I_2)$ we shall use \eqref{eq:quadraticbound} and obtain the following estimate:
    \eqlabel{eq:supestimate1}{\begin{aligned}
        &\int_{[-1,1]\setminus(I_1\cup I_2)}\sup_{v\in \cT}\kappa(v)^{-2\del}\|a_su_rv\|^{-1-\del}dr\\&\leq 48^{1+\del}\left(\sup_{v\in \cT}\phi_\del(v)\right)e^{-(1+\del)s}\int_{[-1,1]\setminus(I_1\cup I_2)}(|r-r_1|^{-2(1+\del)}+|r-r_2|^{-2(1+\del)})dr\\&\leq 200\left(\sup_{v\in \cT}\phi_\del(v)\right)e^{-(1+\del)s}\int_{10e^{-s}<|r|<2}|r|^{-2(1+\del)}dr\\&\leq 400\left(\sup_{v\in \cT}\phi_\del(v)\right)e^{-(1+\del)s}(10e^{-s})^{-(1+2\del)}\leq 40e^{\del s}\sup_{v\in \cT}\phi_\del(v).
    \end{aligned}}

    On the other hand, for $r\in I_1\cup I_2$ we shall use \eqref{eq:Lipschitz0} and get
    \eqlabel{eq:supestimate2}{\begin{aligned}
        \int_{I_1\cup I_2}\sup_{v\in \cT}\kappa(v)^{-2\del}\|a_su_rv\|^{-1-\del}dr&\leq (3e^s)^{1+\del}|I_1\cup I_2|\sup_{v\in \cT}\phi_\del(v)\\&\leq 160e^{\del s}\sup_{v\in \cT}\phi_\del(v).
    \end{aligned}}
    Combining \eqref{eq:kappatrivialbound}, \eqref{eq:supestimate1}, and \eqref{eq:supestimate2}, we obtain the desired inequality.
    
\end{proof}


\subsection{Proof of the subharmonic estimate for $\widetilde{\alpha}_{\del,\eta,M}$}
For a lattice $\Lambda$ in $\bR^3$, denote by $\Lambda_{\operatorname{prim}}$ the set of representative primitive vectors in $\Lambda$ for scalar multiplication by $\pm1$. Note that the height functions $\kappa(v),\phi_\delta(v),\widehat{\phi}_{\delta,\eta,M}(v)$ are all invariant under the scalar multiplication by $-1$. For $0<\del\leq 0.01$, $s\geq 1$, $0<\eta\leq 1$, $M>1$, $g\in H$, and $\Del\in X$ we denote by $\cP_{1,\del,\eta,M}(g,\Del,s)$ the set of $v\in \Del_{\operatorname{prim}}\setminus \cH_{\eta,M}$
satisfying
$$\|gv\|\leq 3e^{s},\quad \phi_{\del}(gv)>(10e^{2s})^{-1-5\del}\widetilde{\alpha}_{1,\del,\eta,M}(g;\Del).$$
Similarly, denote by $\cP_{2,\del,\eta,M}(g,\Del,s)$ the set of $v\in \Del^*_{\operatorname{prim}}\setminus \cH_{\eta,M}$ satisfying
$$\|g^*v\|\leq 3e^{s},\quad \phi_{\del}^*(g^*v)>(10e^{2s})^{-1-5\del}\widetilde{\alpha}_{2,\del,\eta,M}(g;\Del).$$
For $1\leq i\leq 2$ let us also denote
    $$\cQ_{i,\del,\eta,M}(g,\Del,s):=\set{v\in \cP_{i,\del,\eta,M}(g,\Del,s): \kappa_i(gv)\geq \eta (3e^{s})^{-60M}},$$
    $$\cT_{i,\del,\eta,M}(g,\Del,s):=\set{v\in \cP_{i,\del,\eta,M}(g,\Del,s): \kappa_i(gv)< \eta (3e^{s})^{-60M}}.$$
Clearly $\cP_{i,\del,\eta,M}(g,\Del,s)=\cQ_{i,\del,\eta,M}(g,\Del,s)\cup \cT_{i,\del,\eta,M}(g,\Del,s)$ for $i=1,2$.

\begin{lem}\label{lem:findingplane}
    Let $s\geq 1$, $0<\eta\leq 1$, $M>1$, and $(g,\Del)\notin \cE_{s,\eta,M}$. For each $i=1,2$, there exists a plane $L_i\subset \bR^3$ such that $\cT_{i,\del,\eta,M}(g,\Del,s)$ is contained in $L_i\cap\Del$.
\end{lem}
\begin{proof}
    Without loss of generality, it suffices to show the statement for $i=1$. 
    
    Suppose that $\widehat{\alpha}(g;\Del)\leq 10^4e^{4s}$. In this case $(g,\Del)\notin \cE_{s,\eta,M}$ means that there is no vector $v\in g\big(\Del\setminus\cH_{\eta,M}\big)$ with $1\leq \|v\|\leq 3e^s$ and $\kappa(v)<\eta (3e^{s})^{-60M}$. Hence $\cT_{1,\del,\eta,M}(g,\Del,s)$ is contained in $g\Del\cap B(1)$. Since the co-volume of $g\Del$ is one, by Minkowski's first theorem there are at most two independent vectors in $g\Del\cap B(1)$. Thus, there exists a plane $L_1$ such that $\cT_{1,\del,\eta,M}(g,\Del,s)$ is contained in $L_1$.

    Suppose that $\widehat{\alpha}(g;\Del)> 10^4e^{4s}$. Then there exist $v\in g\big(\Del\setminus\cH_{\eta,M}\big)$ such that $\|v\|\leq \frac{1}{50}e^{-2s}$. If there are $v',v''\in g\Del\cap B(3e^s)$ such that $\set{v,v',v''}$ is linearly independent, then this contradicts that the co-volume of $g\Del$ is one. Hence, there exists a plane $L_1$ such that $\cT_{1,\del,\eta,M}(g,\Del,s)$ is contained in $L_1$.
\end{proof}

Now we prove the subharmonic estimate for $\widetilde{\alpha}_{\del,\eta,M}$ in Proposition~\ref{prop:contractionhypothesis}.
\begin{proof}[Proof of Proposition \ref{prop:contractionhypothesis}]
    We first establish the subharmonic estimate for $\widetilde{\alpha}_{1,\del,\eta,M}$.

    \textbf{Case 1}. $|\cQ_{1,\del,\eta,M}(g,\Del,s)|\geq 2$.
    
    Let $v$ and $w$ be linearly independent vectors in $\cQ_{1,\del,\eta,M}(g,\Del,s)$. We have
    \eq{\begin{aligned}
        (10e^{2s})^{-1-5\del}\widetilde{\alpha}_{1,\del,\eta,M}(g;\Del)&<\phi_\del(gv)=\kappa(gv)^{-2\del}\|gv\|^{-1-\del}\\&\leq \eta^{-2\del}e^{2(M+2)\del s}\|gv\|^{-1-\del},
    \end{aligned}}
    hence 
    \eq{\begin{aligned}
        \|gv\|&\leq 10^{1+5\del}\eta^{-2\del}e^{\{2+2(M+7)\del\}s}\big(\widetilde{\alpha}_{1,\del,\eta,M}(g;\Del)\big)^{-1}\\&<10^{1+5\del}\eta^{-2\del}e^{2.5s}\big(\widetilde{\alpha}_{1,\del,\eta,M}(g;\Del)\big)^{-1}.
    \end{aligned}}
    It holds that $\|gw\|\leq 10^{1+5\del}\eta^{-2\del}e^{2.5s}\big(\widetilde{\alpha}_{1,\del,\eta,M}(g;\Del)\big)^{-1}$ similarly. Thus,
    $$\|gv\wedge gw\|\leq\|gv\|\|gw\|\leq 100^{1+5\del}\eta^{-4\del}e^{5s}\big(\widetilde{\alpha}_{1,\del,\eta,M}(g;\Del)\big)^{-2}.$$
    Identifying $gv\wedge gw$ with $gv\times gw\in g^*\Del^*$, we get
    $$\alpha_2(g\Del)^{-1}\leq 100^{1+5\del}\eta^{-4\del}e^{5s}\big(\widetilde{\alpha}_{1,\del,\eta,M}(g;\Del)\big)^{-2}.$$ Using \eqref{eq:Lipschitz}, it follows that for any $r\in[-1,1]$
    $$\widetilde{\alpha}_{1,\del,\eta,M}(a_su_rg;\Del)\leq (3e^s)^{1+\del}\widetilde{\alpha}_{1,\del,\eta,M}(g;\Del)\leq 100\eta^{-2\del}e^{4s}\alpha_2(g\Del)^{\frac{1}{2}}.$$

    \textbf{Case 2}. $|\cQ_{1,\del,\eta,M}(g,\Del,s)|\leq 1$.

    Observe that for any $v\in \Del_{\operatorname{prim}}\setminus\big(\cP_{1,\del,\eta,M}(g,\Del,s)\cup \cH_{\eta,M}\big)$ and $r\in[-1,1]$, we have
    \eqlabel{eq:Pdelta}{\phi_{\delta}(a_su_rgv)< 0.9\widetilde{\alpha}_{1,\del,\eta,M}(a_su_rg;\Del)} by the log-Lipschitz properties \eqref{eq:Lipschitz} and \eqref{eq:Lipschitz'}. Indeed, if $\kappa(a_su_rgv)<1$ and $v\notin \cH_{\eta,M}$ then 
    \eq{\begin{aligned}
        \phi_{\delta}(a_su_rgv)\leq (3e^s)^{1+5\del}\phi_{\delta}(gv)&<\left(\frac{3e^{-s}}{10}\right)^{1+5\del}\widetilde{\alpha}_{1,\del,\eta,M}(g;\Del)\\&<0.9\widetilde{\alpha}_{1,\del,\eta,M}(a_su_rg;\Del),
    \end{aligned}}
    and if $\kappa(a_su_rgv)=1$ and $v\notin\cH_{\eta,M}$ then
    \eq{\begin{aligned}
        \phi_{\delta}(a_su_rgv)&=\|a_su_rgv\|^{-1-\del}\leq (3e^s)^{1+\del}\|gv\|^{-1-\del}\leq(3e^s)^{1+\del}\phi_{\delta}(gv)\\&<\left(\frac{3e^{-s}}{10}\right)^{1+5\del}\widetilde{\alpha}_{1,\del,\eta,M}(g;\Del)<0.9\widetilde{\alpha}_{1,\del,\eta,M}(a_su_rg;\Del).
    \end{aligned}} 
    
    It follows from \eqref{eq:Pdelta} that
    \eqlabel{eq:QTdecomposition}{\begin{aligned}
        \widetilde{\alpha}_{1,\del,\eta,M}(a_su_rg;\Del)&\leq \alpha(a_su_rg\Del)^{1-3\del}+\!\!\!\sup_{v\in \cP_{1,\del,\eta,M}(g,\Del,s)}\!\!\!\widehat{\phi}_{\del,\eta,M}(a_su_rg;v)\\&\leq \alpha(a_su_rg\Del)^{1-3\del}+\!\!\!\sum_{v\in \cQ_{1,\del,\eta,M}(g,\Del,s)}\!\!\!\widehat{\phi}_{\del,\eta,M}(a_su_rg;v)\\&\qquad\qquad\qquad\qquad+\!\!\!\sup_{v\in \cT_{1,\del,\eta,M}(g,\Del,s)}\!\!\!\widehat{\phi}_{\del,\eta,M}(a_su_rg;v).
    \end{aligned}}
    Recall the subharmonic estimate for $\alpha^{1-3\del}$ from Lemma~\ref{eq:triviallinearcontraction1}: we have
    \eqlabel{eq:alphacontribution}{\begin{aligned}
        \int_{-1}^{1}\alpha(a_su_rg\Delta)^{1-3\del}dr&\leq 100e^{-\del s}\alpha(g\Del)^{1-3\del}+e^{4s}\\&\leq 100e^{-\del s}\widetilde{\alpha}_{\del,\eta,M}(g;\Del)+e^{4s}.
    \end{aligned}}
    
    Applying the contraction inequality for $\widehat{\phi}_{\del,\eta,M}$ in Proposition~\ref{eq:linearcontractionhypthesishat}, we also have
    $$\sum_{v\in \cQ_{1,\del,\eta,M}(g,\Del,s)}\int_{-1}^{1}\widehat{\phi}_{\del,\eta,M}(a_su_rg;v)dr\leq  80\del^{-1} e^{-\del s}\!\!\!\sup_{v\in \cQ_{1,\del,\eta,M}(g,\Del,s)}\!\!\!\widehat{\phi}_{\del,\eta,M}(g;v)$$
    since $|\cQ_{1,\del,\eta,M}(g,\Del,s)|\leq 1$. For $v\in \cQ_{1,\del,\eta,M}(g,\Del,s)$, if $1<\|gv\|\leq 3e^s$ then
    \eq{\begin{aligned}
        \widehat{\phi}_{\del,\eta,M}(g;v)=\phi_\del(gv)&=\kappa(gv)^{-2\del}\|gv\|^{-1-\del}\\&\leq \kappa(gv)^{-2\del}\leq \eta^{-2\del}e^{100\del Ms}\leq \eta^{-2\del}e^{s},
    \end{aligned}} and if $\|gv\|\leq 1$ then
    $\widehat{\phi}_{\del,\eta,M}(g;v)\leq \widetilde{\alpha}_{1,\del,\eta,M}(g;\Del)$. Hence,
    \eqlabel{eq:Qcontraction}{\sum_{v\in \cQ_{1,\del,\eta,M}(g,\Del,s)}\int_{-1}^{1}\widehat{\phi}_{\del,\eta,M}(a_su_rg;v)dr\leq  80\del^{-1} e^{-\del s}(\widetilde{\alpha}_{1,\del,\eta,M}(g;\Del)+\eta^{-2\del}e^s).}

    According to Lemma~\ref{lem:findingplane} there exists a plane $L_1$ containing $\cT_{1,\del,\eta,M}(g,\Del,s)$. 
    This allows us to apply the sup-version of the contraction inequality for $\widehat{\phi}_{\del,\eta,M}$ in Proposition~\ref{eq:refinedcontractionhypothesis} and we get
    $$\int_{-1}^{1}\sup_{v\in \cT_{1,\del,\eta,M}(g,\Del,s)}\phi_\del(a_su_rgv)dr\leq 200e^{-\del s}\sup_{v\in \cT_{1,\del,\eta,M}(g,\Del,s)}\phi_\del(gv).$$
    
    Recall that $\phi_\del(gv)\leq \widetilde{\alpha}_{1,\del,\eta,M}(g;\Del)$ for any $gv\in \big(g(\Del\setminus \cH_{\eta,M})\big)^{\circ}$. By our assumption $(g,\Delta)\notin\mathcal{E}_{s,\eta,M}$ we get $\kappa(gv)\geq \eps_{s,\eta,M}(g;\Delta)$ for any $v\in \Delta\setminus \cH_{\eta,M}$ with $1\leq \|gv\|\leq 3e^s$. If $\widehat{\alpha}_{\eta,M}(g;\Del)\leq 10^4e^{4s}$ then every $v\in\cT_{1,\del,\eta,M}(g,\Del,s)$ is contained in $\big(g(\Del\setminus \cH_{\eta,M})\big)^{\circ}$ (see the first case of the proof of Lemma~\ref{lem:findingplane}). If $\widehat{\alpha}_{\eta,M}(g;\Del)> 10^4e^{4s}$ then for any $v\in \cT_{1,\del,\eta,M}(g,\Del,s)$ with $1\leq \|gv\|\leq3e^s$ we have
    $$\phi_\del(gv)\leq \kappa(gv)^{-2\del}\leq \eps_{s,\eta,M}(g;\Del)^{-2\del}=\eta^{-2\del}\alpha(g\Del)^{200DM^2\del}\leq \eta^{-2\del}\alpha(g\Del)^{\frac{1}{2}},$$
    by the definitions \eqref{eq:epsdef} and \eqref{eq:cEdef}. We thus have either $\phi_\del(gv)\leq \eta^{-2\del}\alpha(g\Del)^{\frac{1}{2}}$ or $gv\in \big(g(\Del\setminus \cH_{\eta,M})\big)^{\circ}$. It follows that $$\displaystyle\sup_{v\in \cT_{1,\del,\eta,M}(g,\Del,s)}\!\!\!\phi_\del(gv)\leq  \widetilde{\alpha}_{1,\del,\eta,M}(g;\Del)+\eta^{-2\del}\alpha(g\Del)^{\frac{1}{2}},$$ hence
    \eqlabel{Tcontraction}{\int_{-1}^{1}\sup_{v\in \cT_{1,\del,\eta,M}(g,\Del,s)}\!\!\!\phi_\del(a_su_rv)dr\leq 200e^{-\del s}(\widetilde{\alpha}_{1,\del,\eta,M}(g;\Del)+\eta^{-2\del}\alpha(g\Del)^{\frac{1}{2}}).}
    
    Together with \eqref{eq:QTdecomposition}, \eqref{eq:alphacontribution}, and \eqref{eq:Qcontraction}, it follows that
    $$\int_{-1}^{1}\widetilde{\alpha}_{1,\del,\eta,M}(a_su_rg;\Del)dr\leq 100\del^{-1}e^{-\del s}(\widetilde{\alpha}_{\del,\eta,M}(g;\Del)+\eta^{-2\del}e^{s}\alpha(g\Del)^{\frac{1}{2}}).$$

    From \textbf{Case 1} and \textbf{Case 2} we deduce that
    \eqlabel{eq:alpha1contraction}{\begin{aligned}
        \int_{-1}^{1}\widetilde{\alpha}_{1,\del,\eta,M}&(a_su_rg;\Del)dr\leq 100\del^{-1}e^{-\del s}\widetilde{\alpha}_{\del,\eta,M}(g;\Del)+200\del^{-1}\eta^{-2\del}e^{4s}\alpha(g\Del)^{\frac{1}{2}}\\&\leq 100\del^{-1}e^{-\del s}\widetilde{\alpha}_{\del,\eta,M}(g;\Del)+10^5\del^{-2}\eta^{-4\del}e^{-\del s}\alpha(g\Del)^{1-3\del}+\tfrac{1}{2}e^{9s}\\&\leq \del^{-8}\eta^{-4\del}e^{-\del s}\widetilde{\alpha}_{\del,\eta,M}(g;\Del)+\tfrac{1}{2}e^{9s}.
    \end{aligned}}
    
By the symmetric relation \eqref{eq:alphatildedualrelation} we can obtain the analogous inequality for $\widetilde{\alpha}_{\del,2}$:
\eqlabel{eq:alpha2contraction}{\begin{aligned}
    \int_{-1}^{1}\widetilde{\alpha}_{2,\del,\eta,M}(a_su_rg;\Del)dr&=\int_{-1}^{1}\widetilde{\alpha}_{1,\del,\eta,M}(a_su_rg;J\Del^*)dr\\&\leq \del^{-8}\eta^{-4\del}e^{-\del s}\widetilde{\alpha}_{\del,\eta,M}(g;J\Del^*)+\tfrac{1}{2}e^{9s}\\&= \del^{-8}\eta^{-4\del}e^{-\del s}\widetilde{\alpha}_{\del,\eta,M}(g;\Del)+\tfrac{1}{2}e^{9s}.
\end{aligned}}
The desired inequality \eqref{eq:alphacontraction} follows from \eqref{eq:alpha1contraction} and \eqref{eq:alpha2contraction}.

\end{proof}

\section{Avoidance estimates}
In the previous section, we established the subharmonic estimate for the modified height function $\widetilde{\alpha}_{\del,\eta,M}$ (Proposition \ref{prop:contractionhypothesis}). However, it does not hold for every point $(g,\Del)\in H\times X$: it is valid only if $(g,\Del)$ does not belong to $\cE_{s,\eta,M}$. The goal of this section is to control the amount of time that the orbit $(a_tu_r,\Del)$ stays within the set $\cE_{s,\eta,M}$ by establishing the following estimate.

\begin{prop}[Avoidance estimate]\label{prop:avoidanceMargulis'}
    There exists an absolute constant $D$ such that the following holds. Let $0<\delta\leq 0.01$ and let $Q$ be an indefinite ternary quadratic form of type $M\geq D$ with $\operatorname{det}Q=1$. Then there exists $0<\eta<1$ such that
$$\int_{-1}^{1}\widehat{\alpha}_{\eta,M}'(a_tu_r;\Del_Q)^{1+\del}\mathds{1}_{\cE_{s,\eta,M}}(a_tu_r;\Del_Q)dr\ll e^{-10s}$$
    for $s\geq 1$ and $t\geq 4DMs$, where the implied constant depends only on $Q$.
\end{prop}


\subsection{Integral quadratic forms and closed orbits}

We begin with an elementary lemma.
\begin{lem}\label{lem:quadraticformclosing}
    There is an absolute constant $D_1>1$ such that the following holds. Let $Q$ be an indefinite ternary quadratic form with $\operatorname{det}Q=1$. For any nonzero integral ternary quadratic form $Q'$ the closed orbit $H\Del'$ associated to $(\operatorname{det}Q')^{-\frac{1}{3}}Q'$ satisfies $\operatorname{Vol}(H\Del')<\|Q'\|^{D_1}$ and
    $$\bd(\Del_Q,H\Del')\ll \|Q-(\operatorname{det}Q')^{-\frac{1}{3}}Q'\|.$$
\end{lem}
\begin{proof}
    The space of indefinite ternary quadratic forms with determinant one is identified with $H\backslash G$ by the map $\iota:Q^g\mapsto Hg$ for $g\in H$. Let us choose $g'\in G$ such that $\iota(Q')=Hg'$. Since the map $Q^g\mapsto Hg$ is a $C^1$-map,
    $$\bd^X(\Del_Q,H\Del')\ll \|Q-(\operatorname{det}Q')^{-\frac{1}{3}}Q'\|.$$
    The discriminant of the closed orbit $H\Del'$ is bounded above by $\operatorname{det}(Q')\leq 6\|Q'\|^3$, and the volume of $H\Del'$ is also bounded above by a polynomial of the discriminant of $H\Del'$ (we refer the reader to \cite{ELMV09}, \cite[\S17.3]{EMV09}, and \cite[\S2.6]{LM14} for discussions on the discriminant and the volume of a closed orbit). It follows that $\operatorname{Vol}(H\Del')<\|Q'\|^{D_1}$ for some constant $D_1>1$.
\end{proof}

Let us denote $X_{\leq \mathsf{h}}:= \set{\Del\in X: \alpha(\Del)\leq \mathsf{h}}$ for $\mathsf{h}>1$, and
$$\cK(s,\eps):=\set{\Del\in X:  \Del\cap\Xi_1(s,\eps)\neq\emptyset}$$
for $s\geq 1$ and $0<\eps<1$.

\begin{lem}\label{lem:isotropicdichotomy}
    Let $\mathsf{h}>1$. For any $\Del\in X_{\leq\mathsf{h}}$, $s\geq 1$, $0<\eps<1$, and $T>100$ at least one of the followings holds:
    \begin{enumerate}
        \item $m_{\bR}\big(\set{r\in[-T,T]: u_r\Del\in \cK(s,\eps)}\big)\leq 100,$
        \item there exists $\Del'\in X$ such that $\operatorname{Vol}(H\Del')\leq (e^s\mathsf{h}T^2)^{20D_1}$ and
        $$\bd^X(\Del,\Del')\ll \eps (e^s\mathsf{h}T^2)^{10}.$$
    \end{enumerate}
\end{lem}
\begin{proof}
    For each $v\in \Del\setminus\set{0}$ we denote
    $I(v):=\set{r\in\bR: u_rv\in\Xi(s,\eps)}$ and let
    $$\Phi:=\displaystyle\bigcup_{v\in \Del\setminus\set{0}}I(v)\cap[-T,T].$$
    Then we have
    \eqlabel{eq:avoidanceinclusion}{\set{r\in[-T,T]: u_r\Del\in \cK(s,\eps)}\subseteq \Phi.}
    Recall that $\rho(u_rv)=\rho(v)-r$ for any $r\in\bR$, and $\kappa(u_rv)<1$ holds only if $|\rho(u_rv)|<2$. It follows that $I(v)\cap [-T,T]$ for any $v\in \Del\setminus\set{0}$ is contained in an interval of length $4$.

    If the set $\Phi$ can be covered by eight intervals of length $10$, then in view of \eqref{eq:avoidanceinclusion} it is clear that (1) holds.

    Otherwise, we may find nine points $r_1,\cdots,r_9\in\Phi$ with $|r_i-r_j|>5$ for any $i\neq j$. We shall show that (2) holds in this case. For each $1\leq i\leq 9$ there exists $\bv_i\in (\Del\cup\Del^{*})\setminus\set{0}$ satisfying $u_{r_i}\bv_i\in \Xi(s,\eps)$. By the pigeonhole principle, without loss of generality, we may assume that $\bv_1,\cdots,\bv_5\in \Del\setminus\set{0}$. Note that $u_{r_i}\bv_i\in \Xi(s,\eps)$ implies that $\kappa_0(\bv_i)=\kappa_0(u_{r_i}\bv_i)<\eps$, $|\rho(\bv_i)-r_i|<2$, and $\|\bv_i\|\leq 3|r_i|^2\|u_{r_i}\bv_i\|\leq 10e^sT^2$ for all $1\leq i\leq 5$. Since $\rho(\bv_1),\cdots,\rho(\bv_5)$ are $1$-separated, in view of Lemma~\ref{lem:intersection} no three of $\bv_1,\cdots,\bv_5$ are on the same plane.

    For $1\leq i\leq 5$ we have 
    $$|Q_0(\bv_i)|=|Q_0(u_{r_i}\bv_i)|\leq \kappa(u_{r_i}\bv_i)\|u_{r_i}\bv_i\|^2\ll e^{2s}\eps.$$
    We may choose $g\in H$ with $\Del=g\bZ^3$ so that $\tfrac{1}{10}\|g^{-1}\|\leq\alpha(\Del)\leq \mathsf{h}$. Let $Q=Q^g$ and $\bm_i=g^{-1}\bv_i$ for $1\leq i\leq 5$. Then $|Q(\bm_i)|\leq \eps$ and $$\|\bm_i\|\leq 3\|g^{-1}\|\|\bv_i\|\leq 300e^s\mathsf{h}T^2$$ for all $1\leq i\leq 5$. Moreover, no three of $\bm_i,\cdots,\bm_5$ are on the same plane. Therefore, by Lemma~\ref{lem:effecivekernel} there exists a nonzero integral ternary quadratic form $Q'$ satisfying $\|Q'\|\leq 10^6(300e^s\mathsf{h}T^2)^{14}$ and
    $$\|Q-(\operatorname{det}Q')^{-\frac{1}{3}} Q'\|\ll \eps(e^s\mathsf{h}T^2)^{10}.$$
    It follows from Lemma~\ref{lem:quadraticformclosing} that there exists $\Del'\in X$ such that $H\Del'$ is closed, $$\operatorname{Vol}(H\Del')\leq \|Q'\|^{D_1}\leq (e^s\mathsf{h}T^2)^{20D_1},$$ and $\bd(\Del,H\Del')\ll \eps(e^s\mathsf{h}T^2)^{10}$, hence (2) holds.
\end{proof}

\subsection{Effective avoidance principle}
We recall an effective avoidance theorem from \cite{SS22} in order to control the amount of time that the orbit $a_tu_r\Del$ stays very close to periodic $H$-orbits. 

For positive real numbers $V>1$, $0<d<\tfrac{1}{2}$, and $\mathsf{h}>1$ we denote
\eq{\Upsilon(V,d):=\bigcup_{\operatorname{Vol}(Hx)<V}\set{y\in X: \bd_X(y,Hx)<d},}
\eq{\Upsilon(V,d,\mathsf{h}):=\Upsilon(V,d)\cup (X\setminus X_{\leq\mathsf{h}}).}

The following effective avoidance principle is a slightly simplified reformulation of a special case of \cite[Theorem 2]{SS22} (see also \cite[Proposition 4.6]{LMW22}).
\begin{thm}\label{thm:effectiveavoidance}
    There exists an absolute constant $D_2>1$ such that the following holds. Let $M>1, C>0$. For an indefinite ternary quadratic form $Q$ of type $M\geq D_2$ with $\operatorname{det}Q=1$ we have
    $$m_{\bR}(\set{r\in[-1,1]: a_tu_r\Del_Q\in \Upsilon(V,CV^{-D_2},\mathsf{h})})\ll V^{-1}+\mathsf{h}^{-\frac{1}{D_2}}$$
    for any $V>1$ and $t\geq D_2 M\log V$, where the implied constant depend only on $Q$ and $C$.
\end{thm}

\subsection{Avoidance estimates}
In this subsection, we prove Proposition \ref{prop:avoidanceMargulis}.

\begin{prop}\label{prop:avoidance}
    There is an absolute constant $D_3>200$ such that the following holds. Let $M>1$. For an indefinite ternary quadratic form $Q$ of type $M\geq D_3$ with $\operatorname{det}Q=1$ we have
    $$m_\bR\big(\set{r\in[-1,1]: a_tu_r\Del_Q\in \cK(s,\eps)}\big)\ll \eps^{\frac{1}{D_3}}$$
    for any $s\geq 1$, $0<\eps<e^{-D_3s}$, and $t\geq M\log\frac{1}{\eps}$ where the implied constant depends only on $Q$.
\end{prop}
\begin{proof}
 Let $\mathsf{h}=\eps^{-\frac{1}{200D_1D_2}}$, $T=e^s\mathsf{h}$, and let $C$ be the implied constant in (2) of Lemma~\ref{lem:isotropicdichotomy}. We may cover the interval $[-1,1]$ by disjoint intervals $J_1,\cdots,J_{N_T}$ of length $Te^{-t}$, where $N_T\asymp T^{-1}e^t$. Let $V=(e^{s}\mathsf{h})^{100D_1}$, and let $\Omega\subset\set{1,\cdots,N_T}$ be the set of $i$ such that $a_tu_r\Del_Q\in \Upsilon (V,CV^{-D_2},\mathsf{h})$ for all $r\in J_i$. We choose $D_3=200D_1D_2^2$, then it holds that $$\log V=100D_1(s+\log\mathsf{h})\leq\left(\frac{100D_1}{D_3}+\frac{100D_1}{200D_1D_2}\right)\log\frac{1}{\eps}\leq \frac{1}{D_2}\log\frac{1}{\eps},$$ hence $t\geq M\log\frac{1}{\eps}\geq D_2M\log V$.
 Applying Theorem~\ref{thm:effectiveavoidance}, we have \eqlabel{eq:periodiccontribution}{|\Omega|Te^{-t}\leq m_{\bR}(\set{r\in[-1,1]: a_tu_r\Del_Q\in \Upsilon(V,CV^{-D_2},\mathsf{h})})\ll V^{-1}+\mathsf{h}^{-\frac{1}{D_2}}.}
On the other hand, for $i\notin \Omega$ we may choose $r_i\in J_i$ with $$a_tu_{r_i}\Del_Q\notin \Upsilon (V,CV^{-D_2},\mathsf{h}).$$ Since $\eps<e^{-D_3s}$ there is no $\Del'\in X$ such that $\operatorname{Vol}(H\Del')\leq (e^s\mathsf{h}T^2)^{20D_1}$ and $\bd(a_tu_{r_i}\Del_Q,\Del')\leq C\eps (e^s\mathsf{h}T^2)^{10}$. Thus (1) of Lemma \ref{lem:isotropicdichotomy} must hold, hence
\eqlabel{eq:isotropiccontribution}{\begin{aligned}
    m_\bR&\big(\set{r\in J_i: a_tu_r\Del_Q\in \cK(s,\eps)}\big)\\&\leq e^{-t}m_{\bR}\big(\set{r\in[-T,T]: u_ra_tu_{r_i}\Del_Q\in \cK(s,\eps)}\big)\leq 100e^{-t}.
\end{aligned}}
Combining \eqref{eq:periodiccontribution} and \eqref{eq:isotropiccontribution}, we get
\eq{\begin{aligned}
    m_\bR\big(\set{r\in[-1,1]: a_tu_r\Del_Q\in \cK(s,\eps)}\big)&\ll \sum_{i\in \Omega} Te^{-t}+\sum_{i\notin\Omega}e^{-t}\\&\leq |\Omega|Te^{-t}+N_Te^{-t}\\&\ll V^{-1}+\mathsf{h}^{-\frac{1}{D_2}}+T^{-1}\ll \eps^{\frac{1}{D_3}}.
\end{aligned}}
\end{proof}

We now prove an avoidance estimate for $\widehat{\alpha}_{\eta,M}$.
\begin{prop}\label{prop:avoidanceMargulis}
    There exists an absolute constant $D\geq 10^4$ such that the following holds. Let $0<\delta\leq 0.01$ and let $Q$ be an indefinite ternary quadratic form of type $M\geq D$ with $\operatorname{det}Q=1$. Then we have
$$\int_{-1}^{1}\widehat{\alpha}_{\eta,M}(a_tu_r;\Del_Q)^{1+\del}\mathds{1}_{\cE_{s,\eta,M}}(a_tu_r;\Del_Q)dr\ll e^{-10s}$$
    for any $0<\eta<1$, $s\geq 1$, and $t\geq 4DMs$, where the implied constant depends only on $Q$.
\end{prop}
\begin{proof}
    We choose $D:=5D_3^2$. Let $\mathsf{h}_i=10^4e^{4s+i}$ for $i\in\bN$. 
    Recall from Lemma \ref{lem:Q0Diophantineexpanding} that if $\widehat{\alpha}_{\eta,M}(a_tu_r,\Delta_Q)>10^4e^{\frac{t}{DM}}$ then $(a_tu_r,\Delta_Q)\notin\mathcal{E}_{1,s,\eta,M}$ for all $r\in[-1,1]$. We thus observe that if $(a_su_r,\Del_Q)\in \cE_{1,s,\eta,M}$ then either one of the followings holds:
    \begin{enumerate}
        \item $a_tu_r\Delta_Q\in \cK(s,e^{-60Ms})$ and $\widehat{\alpha}_{\eta,M}(a_tu_r,\Delta_Q)\leq 10^4e^{4s}$, or
        \item $a_tu_r\Delta_Q\in \cK(s,\mathsf{h}_i^{-5D_3})$ and $\widehat{\alpha}_{\eta,M}(a_tu_r,\Delta_Q)\leq \mathsf{h}_i$ for some $1\leq i\leq \lceil\frac{t}{DM}\rceil$.
    \end{enumerate}
    We shall estimate
    $$\mathsf{L}_0:=m_\bR\left(\set{r\in[-1,1]: a_tu_r\Delta_Q\in \cK(s,e^{-60Ms}) \textrm{ and }\widehat{\alpha}_{\eta,M}(a_tu_r,\Delta_Q)\leq 10^4e^{4s}}\right),$$
    $$\mathsf{L}_i:=m_\bR\big(\set{r\in[-1,1]: a_tu_r\Delta_Q\in \cK(s,\mathsf{h}_i^{-5D_3})\textrm{ and }\widehat{\alpha}_{\eta,M}(a_tu_r,\Delta_Q)\leq \mathsf{h}_i}\big)$$
    for $1\leq i\leq \lceil\frac{t}{DM}\rceil$, using Proposition \ref{prop:avoidance}. We first check if the sets $\cK(s,e^{-60Ms})$ and $\cK(s,\mathsf{h}_i^{-5D_3})$ satisfy the assumptions in Proposition \ref{prop:avoidance}. For $\eps=e^{-60Ms}<e^{-Ds}$ the assumptions are satisfied since $t\geq 4DMs$, and for $\eps=\mathsf{h}_i^{-5D_3}$ the assumptions are satisfied so long as $1\leq i\leq \lceil\frac{t}{DM}\rceil$. Indeed, $\mathsf{h}_i^{-5D_3}\leq (10^4e^{4s})^{-5D_3}<e^{-D_3s}$ holds for all $i\in\bN$ and $t\geq M\log(\mathsf{h}_i^{5D_3})$ holds for $1\leq i\leq \lceil\frac{t}{DM}\rceil$.
    
    Applying Proposition \ref{prop:avoidance} we have
    $\mathsf{L}_0\ll e^{-\frac{D}{D_3}s}=e^{-5D_3s}$, and $\mathsf{L}_i\ll \mathsf{h}_i^{-5}$ for $1\leq i\leq \lceil\frac{t}{DM}\rceil$. We conclude from these estimates that
    \eq{\begin{aligned}
        \int_{-1}^{1}\widehat{\alpha}_{\eta,M}(a_tu_r;\Del_Q)^{1+\del}\mathds{1}_{\cE_{1,s,\eta,M}}(a_tu_r;\Del_Q)dr&\leq (10^4e^{4s})^{1+\del}\mathsf{L}_0+\sum_{i=1}^{\lceil\frac{t}{DM}\rceil}\mathsf{h}_i^{1+\del}\mathsf{L}_i\\&\ll e^{-(5D_3-5)s}+\sum_{i=1}^{\infty}\mathsf{h}_i^{-3}
        \\&\ll e^{-10s}+\mathsf{h}_1^{-3}\ll e^{-10s}.
    \end{aligned}}
    By the relations \eqref{eq:alphatildedualrelation} and $\cE_{1,s,\eta,M}=\set{(g,J\Delta^*):(g;\Delta)\in\cE_{2,s,\eta,M}}$, this also implies
    \eq{\begin{aligned}
    \int_{-1}^{1}&\widehat{\alpha}_{\eta,M}(a_tu_r;\Del_Q)^{1+\del}\mathds{1}_{\cE_{2,s,\eta,M}}(a_tu_r;\Del_Q)dr\\&=\int_{-1}^{1}\widehat{\alpha}_{\eta,M}(a_tu_r;J\Del_Q^*)^{1+\del}\mathds{1}_{\cE_{1,s,\eta,M}}(a_tu_r;J\Del_Q^*)dr\ll e^{-10s},
\end{aligned}}
hence
$$\int_{-1}^{1}\widehat{\alpha}_{\eta,M}(a_tu_r;\Del_Q)^{1+\del}\mathds{1}_{\cE_{s,\eta,M}}(a_tu_r;\Del_Q)dr\ll e^{-10s}.$$
\end{proof} 

We deduce Proposition~\ref{prop:avoidanceMargulis'} from Proposition~\ref{prop:avoidanceMargulis}.

\begin{proof}[Proof of Proposition~\ref{prop:avoidanceMargulis'}]
By the definition of $\widehat{\alpha}_{\eta,M}'$, we have
\eq{\begin{aligned}
    \widehat{\alpha}_{\eta,M}'(a_tu_r;\Del_Q)^{1+\delta}&\leq \max\set{\widehat{\alpha}_{\eta,M}(a_tu_r;\Delta)^{1+\delta},\alpha(a_tu_r\Delta)^{0.9(1+\delta)}}\\&\leq \widehat{\alpha}_{\eta,M}(a_tu_r;\Del_Q)^{1+\delta}+\!\!\!\!\!\!\sup_{v\in\Delta_Q\cap\cH_{\eta,M}}\!\!\!\!\!\!\|a_tu_rv\|^{-0.9(1+\delta)}+\!\!\!\!\!\!\sup_{v\in\Delta_Q^*\cap\cH_{\eta,M}}\!\!\!\!\!\!\|a_t^*u_r^*v\|^{-0.9(1+\delta)}.
\end{aligned}}
Hence, together with Proposition~\ref{prop:avoidanceMargulis} it suffices to show that 
$$\int_{-1}^{1}\sup_{v\in\Delta_Q\cap\cH_{\eta,M}}\!\!\!\!\|a_tu_rv\|^{-0.9(1+\delta)}dr\ll e^{-10s},$$
$$\int_{-1}^{1}\sup_{v\in\Delta_Q^*\cap\cH_{\eta,M}}\!\!\!\!\|a_t^*u_r^*v\|^{-0.9(1+\delta)}dr\ll e^{-10s}$$
hold for some $0<\eta<1$. 

Recall from Lemma \ref{lem:twelvelines} that there exists $0<\eta<1$ such that for any $R>10$ the set $\set{v\in (\Delta_Q\cup\Delta_Q^*)\cap \cH_{\eta,M}: R\leq \|v\|<R^2}$ is contained in at most $12$ lines. For $0\leq R<R'\leq \infty$ let us denote
    \eqlabel{eq:Pidef}{\Pi(R,R'):=\set{v\in \Del_Q\cap \cH_{\eta,M}: R< \|v\|\leq R'},}
    \eqlabel{eq:Pistardef}{\Pi^*(R,R'):=\set{v\in \Del_Q^*\cap \cH_{\eta,M}: R< \|v\|\leq R'}.}

    Applying Lemma~\ref{eq:triviallinearcontraction0} with $\lambda=0.9(1+\delta)$ for primitive vectors of each line, we get
    $$\int_{-1}^{1}\sup_{v\in \Pi(R,R^2)}\!\!\!\!\|a_tu_rv\|^{-0.9(1+\delta)}dr\leq 2400e^{-0.01 t}\!\!\!\!\sup_{v\in \Pi(R,R^2)}\!\!\!\!\|v\|^{-0.9(1+\delta)}\leq 2400e^{-0.01 t}R^{-1}$$
    for any $R>10$. It follows that for any $R_0>10$
    \eqlabel{eq:Largepicontribution}{\begin{aligned}
        \int_{-1}^{1}\sup_{v\in \Pi(R_0,\infty)}\!\!\!\!\|a_tu_rv\|^{-0.9(1+\delta)}dr&\leq \sum_{i=0}^{\infty}\int_{-1}^{1}\sup_{v\in \Pi(R_0^{2^i},R_0^{2^{i+1}})}\!\!\!\!\|a_tu_rv\|^{-0.9(1+\delta)}dr\\&\leq 2400e^{-0.01 t}\sum_{i=0}^{\infty}R_0^{-0.9\cdot2^{i}}\leq 10^4R_0^{-0.9}e^{-0.01 t}.
    \end{aligned}}
    Similarly, we also find that
    \eqlabel{eq:Largepistarcontribution}{\int_{-1}^{1}\sup_{v\in \Pi^*(R_0,\infty)}\!\!\!\!\|a_t^*u_r^*v\|^{-0.9(1+\delta)}dr\leq 10^4R_0^{-0.9}e^{-0.01 t}.}

 Since there are only finitely many vectors in $\Pi^*(0,R_0)$ and $\Pi^*(0,R_0)$, we conclude that $$\int_{-1}^{1}\sup_{v\in\Delta_Q\cap\cH_{\eta,M}}\!\!\!\!\|a_tu_rv\|^{-0.9(1+\delta)}dr\ll e^{-0.01 t}\ll e^{-10s},$$
$$\int_{-1}^{1}\sup_{v\in\Delta_Q^*\cap\cH_{\eta,M}}\!\!\!\!\|a_t^*u_r^*v\|^{-0.9(1+\delta)}dr\ll e^{-0.01 t}\ll e^{-10s}$$
 holds, where the implied constant depends only on $Q$. This completes the proof.
\end{proof}

\section{Proof of the moment estimate}
In this section, we prove the moment estimate in Theorem \ref{thm:mainthm}. We denote by $m_I$ the uniform Lebesgue probability measure on the interval $I=[-1,1]$. The next lemma allows us to write expanding translates of a unipotent orbit in the form of an iteration of random walks.


\begin{lem}\label{lem:randomwalkgeneral}
    Let $m\in\bN$ and $s_1,\cdots,s_m\geq 1$. Then
    $$\int_{-\frac{1}{3}}^{\frac{1}{3}} f (a_{s_1+\cdots+s_m}u_r;\Del)dr\leq 2\int f(a_{s_m}u_{r_m}\cdots a_{s_1}u_{r_1};\Del)dm_I^{\otimes m}(r_1,\cdots,r_m),$$
    $$\int f(a_{s_m}u_{r_m}\cdots a_{s_1}u_{r_1};\Del)dm_I^{\otimes m}(r_1,\cdots,r_m)\leq \int_{-2}^{2} f (a_{s_1+\cdots+s_m}u_r;\Del)dr$$
    for any non-negative measurable function $f:H\times X\to [0,\infty)$ and $\Del\in X$.
\end{lem}
\begin{proof}
    Let $t_i=s_1+\cdots+s_{i-1}$ and $\omega_i=\frac{e^{t_i}}{2}\mathds{1}_{[-e^{-t_i},e^{-t_i}]}$ for $1\leq i\leq m$, and let $\omega=\omega_1*\cdots*\omega_m$. Notice that
    $$\int f(a_{s_m}u_{r_m}\cdots a_{s_1}u_{r_1};\Del)dm_I^{\otimes m}(r_1,\cdots,r_m)=\int f(a_{s_1+\cdots+s_m}u_{r};\Del)\omega(r)dr.$$
    By induction, one can check that for any $1\leq k\leq m$,
    $$0\leq \omega_1*\cdots*\omega_k(r)\leq \frac{1}{2} \textrm{ for all } r\in\bR,$$
    $$\omega_1*\cdots*\omega_k(r)=\frac{1}{2} \textrm{ if } |r|<1-\sum_{i=1}^{k-1}e^{-t_i},$$
    $$\omega_1*\cdots*\omega_k(r)=0 \textrm{ if } |r|>1+\sum_{i=1}^{k-1}e^{-t_i}.$$ Thus, we have $0\leq \omega(r)\leq \frac{1}{2}$ for all $r\in\bR$, $\omega(r)=\frac{1}{2}$ if $|r|<\frac{1}{3}$, and $\omega(r)=0$ if $|r|>2$. It follows that
    \eq{\int f(a_{s_1+\cdots+s_m}u_{r};\Del)\omega(r)dr\geq \frac{1}{2}\int_{-\frac{1}{3}}^{\frac{1}{3}} f(a_{s_1+\cdots+s_m}u_{r};\Del)dr,}
    \eq{\int f(a_{s_1+\cdots+s_m}u_{r};\Del)\omega(r)dr\leq \int_{-2}^{2} f(a_{s_1+\cdots+s_m}u_{r};\Del)dr,}
    completing the proof.
\end{proof}

We will also need the following elementary lemma.
\begin{lem}\label{lem:ndecomposition}
    Let $B>1$, $0<\del<\frac{1}{1+B}$, and $T>0$ be given. For any $t\geq \del^{-1}T$, we can find a finite sequence $\set{s_i}_{1\leq i\leq N}$ such that $t=s_1+\cdots+s_N$, $s_1=Bs_2$, $s_i=(1+\del)s_{i+1}$ for $2\leq i\leq N-1$, and $T\leq s_N\leq 2T$.
\end{lem}
\begin{proof}
    Let us denote
    $$b_k=B(1+\del)^k+\displaystyle\sum_{i=1}^{k}(1+\del)^i=(\del^{-1}+B+1)(1+\del)^k-\del^{-1}$$ for $k\geq 0$. Since $B>1$ and $\del^{-1}>1+B$, we have $b_{k-1}\leq b_{k}\leq 2b_{k-1}$ for any $k\geq 1$. Thus, we can find $T\leq \tau \leq 2T$ and $k\geq 0$ such that $t=b_k\tau$. Set $N=k+2$, $s_1=B(1+\del)^k\tau$, and $s_i=(1+\del)^{N-i}\tau$ for $2\leq i\leq N$. Then $\set{s_i}_{1\leq i\leq N}$ satisfies the conditions.
\end{proof}

We are now ready to prove Theorem \ref{thm:mainthm}.
\begin{proof}[Proof of Theorem \ref{thm:mainthm}]
    Without loss of generality, it suffices to show that
    $$\sup_{t>0}\int_{-\frac{1}{3}}^{\frac{1}{3}} \widehat{\alpha}_{\eta,M}(a_tu_r;\Del_Q)^{1+\del}dr<\infty.$$
    Since the $M$-Diophantine condition implies the $D$-Diophantine condition for $M<D$, we may assume $M\geq D$. Let us set
    $$B=4DM,$$ $$\del=\frac{1}{10^5BDM(M+7)}=\frac{1}{4\cdot 10^5D^2M^2(M+7)},$$ $$\delta'=40B\delta,$$ $$T=100\del^{-2}-10\log \eta.$$ 
    Note that $400\del^{-1}\eta^{-4\del}\leq e^{\frac{\del}{2}s}$ for any $s\geq T$, and $\del<\del'<\frac{1}{400DM(M+7)}$. We now recall the superharmonic estimate for $\widehat{\alpha}_{\eta,M}'$ in Lemma~\ref{eq:trivialcontractionhypothesishat} and the subharmonic estimate for $\widetilde{\alpha}_{\del',\eta,M}$ in Proposition~\ref{prop:contractionhypothesis}. We have
    \eqlabel{eq:recalltrivialcontraction}{\begin{aligned}
        \int_{-1}^{1}\widehat{\alpha}_{\eta,M}'(a_su_r;\Del)^{1+\del}dr&\leq 400\del^{-1}e^{\del s}\alpha(\Del)^{1+\del}+e^{6s}\\&\leq e^{2\del s}\alpha(\Del)^{1+\del}+e^{6s}
    \end{aligned}}
    for any $\Del\in X$ and $s\geq T$, and
    \eqlabel{eq:recallnontrivialcontraction}{\begin{aligned}
        \int_{-1}^{1}\widetilde{\alpha}_{\del',\eta,M}(a_su_rg;\Del)dr&\leq (400\del'^{-1}\eta^{-4\del})e^{-\del' s}\widetilde{\alpha}_{\del',\eta,M}(g;\Del)+e^{9s}\\&\leq e^{-\frac{\del'}{2}s}\widetilde{\alpha}_{\del',\eta,M}(g;\Del)+e^{9s}
    \end{aligned}}
    for any $(g,\Del)\notin\cE_{s,\eta,M}$ and $s\geq T$.
    
    Given $t\ge \del^{-1}T$, we can find a finite sequence $\set{s_i}_{1\leq i\leq N}$ as in Lemma~\ref{lem:ndecomposition}. Then the sequence $\set{s_i}_{1\leq i\leq N}$ satisfies
    \eqlabel{eq:sibound}{t=s_1+\cdots+s_N,\qquad s_1=Bs_2,\qquad T\leq s_N\leq 2T,}
    \eqlabel{eq:sumratio}{\left(1-\frac{s_N}{s_i}\right)\del^{-1}s_i\leq s_{i+1}+\cdots+s_N\leq \del^{-1}s_i \textrm{ for any }2\leq i\leq N-1.}

    We consider
    $$\mathsf{Z}_t:=\int \widehat{\alpha}_{\eta,M}(a_{s_{N}}u_{r_{N}}\cdots a_{s_1}u_{r_1};\Del_Q)^{1+\del}dm_I^{\otimes N}(r_1,\cdots,r_{N}).$$
    From now on we shall prove that $\displaystyle\sup_{t>0}\mathsf{Z}_t<\infty$ since Lemma~\ref{lem:randomwalkgeneral} gives that
    \eqlabel{eq:Ztexpression}{\int_{-\frac{1}{3}}^{\frac{1}{3}} \widehat{\alpha}_{\eta,M}(a_tu_r;\Del_Q)^{1+\del}dr\leq 2\mathsf{Z}_t.}
    Let us define
    $$\overline{\Theta}_m:=\set{(r_1,\ldots,r_m)\in I^m: (a_{s_m}u_{r_m}\cdots a_{s_1}u_{r_1},\Del_Q)\in \cE_{s_{m+1},\eta,M}},$$
    $$\Theta_m:=\overline{\Theta}_m\times I^{N-m}\subseteq I^{N}, \qquad \Theta:=\bigcup_{m=1}^{N}\Theta_m$$
    for $1\leq m\leq N-1$. We shall estimate $\mathsf{Z}_t$ by $\mathsf{Z}_t\leq \mathsf{Y}_t+\sum_{m=1}^{N-1}\mathsf{E}_{t,m}$, where
    $$\mathsf{Y}_t:=\int_{I^{N}\setminus\Theta} \widehat{\alpha}_{\eta,M}(a_{s_{N}}u_{r_{N}}\cdots a_{s_1}u_{r_1};\Del_Q)^{1+\del}dm_I^{\otimes N}(r_1,\ldots,r_{N}),$$
    $$\mathsf{E}_{t,m}:=\int_{\Theta_m} \widehat{\alpha}_{\eta,M}(a_{s_{N}}u_{r_{N}}\cdots a_{s_1}u_{r_1};\Del_Q)^{1+\del}dm_I^{\otimes N}(r_1,\ldots,r_{N})$$
    for $1\leq m\leq N-1$. 
    
    We shall first estimate $\mathsf{E}_{t,m}$. Write
    $$\mathsf{E}_{t,m}=\int_{\overline{\Theta}_m}J(r_1,\ldots,r_m)dm_I^{\otimes m}(r_1,\ldots,r_m),$$
    where
    $$J(r_1,\ldots,r_m):=\int_{I^{N-m}}\widehat{\alpha}_{\eta,M}(a_{s_{N}}u_{r_{N}}\cdots a_{s_1}u_{r_1};\Del_Q)^{1+\del}dm_I^{\otimes (N-m)}(r_{m+1},\ldots,r_{N}).$$
    
    Since $s_i\geq T$ for any $1\leq i\leq N$, we may apply \eqref{eq:recalltrivialcontraction} for $J(r_1,\cdots,r_m)$ repeatedly with $s=s_N,\ldots,s_{m+1}$, and get
    \eqlabel{eq:Badestimate1}{\begin{aligned}
        J&(r_1,\ldots,r_m)\\&\leq \int_{I^{N-m}}\widehat{\alpha}_{\eta,M}'(a_{s_{N}}u_{r_{N}}\cdots a_{s_1}u_{r_1};\Del_Q)^{1+\del}dm_I^{\otimes (N-m)}(r_{m+1},\ldots,r_{N})\\&\leq \sum_{i=m+1}^{N}e^{6s_i}e^{2\del(s_{i+1}+\cdots+s_{N})}\widehat{\alpha}_{\eta,M}'(a_{s_m}u_{r_m}\cdots a_{s_1}u_{r_1}\Del_Q)^{1+\del}\\&\leq (N-m)e^{8s_{m+1}}\widehat{\alpha}_{\eta,M}'(a_{s_m}u_{r_m}\cdots a_{s_1}u_{r_1}\Del_Q)^{1+\del}\\&\leq 2\del^{-1}\log s_{m+1} e^{8s_{m+1}}\widehat{\alpha}_{\eta,M}'(a_{s_m}u_{r_m}\cdots a_{s_1}u_{r_1}\Del_Q)^{1+\del},
    \end{aligned}}
    using \eqref{eq:sumratio} in the penultimate inequality.
    

    We let $\varphi(g;\Del):=\widehat{\alpha}_{\eta,M}'(g;\Del)^{1+\del}\cdot\mathds{1}_{\cE_{s_{m+1},\eta,M}}(g,\Del)$ for $(g,\Del)\in H\times X$. By Proposition~\ref{prop:avoidanceMargulis'} and Lemma~\ref{lem:randomwalkgeneral} we have
    \eq{\begin{aligned}
        \int_{\overline{\Theta}_m}\widehat{\alpha}_{\eta,M}'&(a_{s_m}u_{r_m}\cdots a_{s_1}u_{r_1};\Del_Q)^{1+\del}dm_I^{\otimes m}(r_1,\ldots,r_m)\\&=\int_{I^m}\varphi(a_{s_m}u_{r_m}\cdots a_{s_1}u_{r_1};\Del_Q)dm_I^{\otimes m}(r_1,\ldots,r_m)\\&\leq \int_{-2}^{2}\varphi(a_{s_1+\cdots+s_m}u_{r};\Del_Q)dr\ll e^{-10s_{m+1}}.
    \end{aligned}}
    
    In combination with \eqref{eq:Badestimate1} it follows that
    \eqlabel{eq:Badestimate3}{\begin{aligned}\mathsf{E}_{t,m}&\leq 2\del^{-1}\log s_{m+1}e^{8s_{m+1}}\int_{\overline{\Theta}_m}\widehat{\alpha}_{\eta,M}'(a_{s_m}u_{r_m}\cdots a_{s_1}u_{r_1};\Del_Q)^{1+\del}dm_I^{\otimes m}(r_1,\ldots,r_m)\\&\ll \del^{-1}\log s_{m+1}e^{8s_{m+1}}e^{-10s_{m+1}}\ll e^{-s_{m+1}}
    \end{aligned}}
    for any $1\leq m\leq N-1$. Therefore, we get
    \eqlabel{eq:Ebound}{\sum_{m=1}^{N-1}\mathsf{E}_{t,m} \ll  \sum_{m=1}^{N-1}e^{-s_{m+1}}< \sum_{m=1}^{\infty} e^{-(1+\del)^mT}<\sum_{m=1}^{\infty} e^{-m\del T}<\del^{-1}T^{-1}.}

    We now estimate $\mathsf{Y}_t$ using \eqref{eq:recallnontrivialcontraction}. Observe that for any $1\leq m\leq N$ and $(r_1,\ldots,r_{N})\in I^{N}\setminus\Theta$ we have 
    $$(a_{s_m}u_{r_m}\cdots a_{s_1}u_{r_1},\Del_Q)\notin \cE_{s_{m+1},\eta,M}$$ from the construction of $\Theta$. By \eqref{eq:identitycontained} we also have $(\operatorname{id},\Del_Q)\notin \cE_{s_1,\eta,M}$.
    
    Applying \eqref{eq:recallnontrivialcontraction} repeatedly with $s=s_N,\ldots,s_1$, we get
    \eqlabel{eq:Goodestimate1}{\begin{aligned}
        \mathsf{Y}_t&<\int_{I^{N}\setminus\Theta} \widehat{\alpha}_{\eta,M}(a_{s_{N}}u_{r_{N}}\cdots a_{s_1}u_{r_1};\Del_Q)^{1+\del'}dm_I^{\otimes N}(r_1,\ldots,r_{N})\\&\leq\int_{I^{N}\setminus\Theta} \widetilde{\alpha}_{\del',\eta,M}(a_{s_{N}}u_{r_{N}}\cdots a_{s_1}u_{r_1};\Del_Q)dm_I^{\otimes N}(r_1,\ldots,r_{N})\\&\leq e^{-\frac{\del'}{2}t}\widetilde{\alpha}_{\del',\eta,M}(\operatorname{id};\Del_Q)+\sum_{m=1}^{N}e^{9s_m-\frac{\del'}{2}(s_{m+1}+\cdots+s_{N})}.
    \end{aligned}}
    Let $N'$ be the largest integer such that $s_{N'}\geq 2s_N$. Note that $N-N'< \del^{-1}$, since $s_{N'}=(1+\del)^{N-N'}s_N$. If $2\leq m\leq N'$, then $s_{m+1}+\cdots+s_{N}\geq \left(1-\frac{s_N}{s_{N'}}\right)\del^{-1}s_m\geq\frac{\del^{-1}s_m}{2}=\frac{20}{\del'}s_m$. Moreover, for $m=1$ we have
    $$\frac{\del'}{2}(s_2+\cdots+s_N)\geq \frac{\del^{-1}\del' s_2}{4}=10Bs_2=10s_1,$$
    hence
    \eqlabel{eq:exponentialsum1}{\sum_{m=1}^{N'}e^{9s_m-\frac{\del'}{2}(s_{m+1}+\cdots+s_{N})}\leq \sum_{m=1}^{N'}e^{-s_m}< \sum_{m=1}^{\infty} e^{-(1+\del)^mT}<\del^{-1}T^{-1}.}
    It is also easy to see
    \eqlabel{eq:exponentialsum2}{\displaystyle\sum_{m=N'+1}^{N}e^{9s_m-\frac{\del'}{2}(s_{m+1}+\cdots+s_{N})}\leq (N-N')e^{18s_N}< \del^{-1}e^{36T}.}
    Combining \eqref{eq:Goodestimate1}, \eqref{eq:exponentialsum1}, and \eqref{eq:exponentialsum2}, we have
    $$\mathsf{Y}_t< e^{-\frac{\del'}{2}t}\widetilde{\alpha}_{\del',\eta,M}(\operatorname{id};\Del_Q)+\del^{-1}e^{36T}+\del^{-1}T^{-1}$$
    for any large $t$. In combination with \eqref{eq:Ebound}, it follows that
    $$\mathsf{Z}_t\leq \mathsf{Y}_t+\sum_{m=1}^{N}\mathsf{E}_{t,m}\leq e^{-\frac{\del'}{2}t}\widetilde{\alpha}_{\del',\eta,M}(\operatorname{id};\Del_Q)+\del^{-1}(e^{36T}+2T^{-1})$$
    for any large $t$, hence $\displaystyle\sup_{t>0}\mathsf{Z}_t<\infty$. By \eqref{eq:Ztexpression} this completes the proof.
\end{proof}

\section{Equidistribution for unbounded test functions}

\subsection{Contribution of non-isotropic quasi-null vectors}
In this subsection, we show that the contribution of non-isotropic quasi-null vectors is small.

\begin{prop}\label{prop:quasinullcontribution}
Let $Q$ be an indefinite quadratic form of Diophantine type $M$ with $\operatorname{det}Q=1$. Then there exists $0<\eta<1$ such that 
     $$\lim_{t\to\infty}\int_{-1}^{1}\sup_{v\in \Del_Q\cap (\cH_{\eta,M}\setminus\cH)}\!\!\!\!\!\!\!\!\|a_tu_rv\|^{-1}+\!\!\!\!\!\!\sup_{v\in \Del_Q^{*}\cap (\cH_{\eta,M}\setminus\cH)}\!\!\!\!\!\!\!\!\|a_t^*u_r^*v\|^{-1}dr=0.$$
\end{prop}
\begin{proof}
    For $0\leq R<R'\leq \infty$ let us denote
    $$\Pi_0(R,R'):=\set{v\in \Del_Q\cap (\cH_{\eta,M}\setminus\cH): R< \|v\|\leq R'},$$
    $$\Pi_0^*(R,R'):=\set{v\in \Del_Q^*\cap (\cH_{\eta,M}\setminus\cH): R< \|v\|\leq R'}.$$

    By Lemma~\ref{lem:twelvelines}, for any $R>10$ the points in $\Pi_0(R,R^2)$ are contained in at most $12$ lines. Applying Lemma~\ref{eq:triviallinearcontraction0} with $\lambda=1$ for primitive vectors of each line, we get
    $$\int_{-1}^{1}\sup_{v\in \Pi_0(R,R^2)}\!\!\!\!\!\!\|a_tu_rv\|^{-1}dr\leq 2400\sup_{v\in \Pi_0(R,R^2)}\!\!\!\!\|v\|^{-1}\leq 2400R^{-1}$$
    for any $R>10$. It follows that for any $R_0>10$
    \eqlabel{eq:Largepi0contribution}{\begin{aligned}
        \int_{-1}^{1}\sup_{v\in \Pi_0(R_0,\infty)}\!\!\!\!\!\!\|a_tu_rv\|^{-1}dr&\leq \sum_{i=0}^{\infty}\int_{-1}^{1}\sup_{v\in \Pi_0(R_0^{2^i},R_0^{2^{i+1}})}\!\!\!\!\!\!\!\!\|a_tu_rv\|^{-1}dr\\&\leq 2400\sum_{i=0}^{\infty}R_0^{-2^{i}}\leq 10^4R_0^{-1}.
    \end{aligned}}
    Similarly, we also find that
    \eqlabel{eq:Largepistarcontribution0}{\int_{-1}^{1}\sup_{v\in \Pi_0^*(R_0,\infty)}\!\!\!\!\!\!\|a_t^*u_r^*v\|^{-1}dr\leq 10^4R_0^{-1}.}
    
    On the other hand, note that for given $R_0$ there are only finitely many points in $\Pi_0(0,R_0)$ and $\Pi_0^*(0,R_0)$. For each $v\notin \cH$ we may view the contraction for $\phi_\del$ in Proposition~\ref{eq:linearcontractionhypothesis} and obtain
    $$\limsup_{t\to\infty}\int_{-1}^{1}\|a_tu_rv\|^{-1}dr\leq \limsup_{t\to\infty}\int_{-1}^{1}\phi_{0.01}(a_tu_rv)dr\leq \limsup_{t\to\infty}10^4e^{-0.01t}\phi_{0.01}(v)=0,$$
    as $\phi_{0.01}(v)<\infty$ for $v\notin \cH$. Thus for any $R_0>0$ we deduce that
    \eqlabel{eq:Smallpicontribution}{\begin{aligned}
        \limsup_{t\to\infty}&\int_{-1}^{1}\sup_{v\in \Pi_0(0,R_0)}\!\!\!\!\|a_tu_rv\|^{-1}dr+\int_{-1}^{1}\sup_{v\in \Pi_0^*(0,R_0)}\!\!\!\!\|a_t^*u_r^*v\|^{-1}dr\\&\leq \limsup_{t\to\infty}\sum_{v\in \Pi_0(0,R_0)}\int_{-1}^{1}\|a_tu_rv\|^{-1}dr+\sum_{v\in \Pi_0^*(0,R_0)}\int_{-1}^{1}\|a_t^*u_r^*v\|^{-1}dr=0.
    \end{aligned}}
    Combining \eqref{eq:Largepi0contribution}, \eqref{eq:Largepistarcontribution0}, and \eqref{eq:Smallpicontribution}, and taking $R_0\to\infty$, we complete the proof.
\end{proof}

\subsection{Proof of Theorem~\ref{prop:unboundedRatner}}
We now derive Theorem~\ref{prop:unboundedRatner} from Theorem~\ref{thm:mainthm} combining with Proposition~\ref{prop:quasinullcontribution}. As explained in \S 2.2.5, Theorem~\ref{prop:unboundedRatner} in turn implies Theorem~\ref{thm:quantitativeoppenheim0}. We assume without loss of generality that $f\geq 0$.

We note that the moment estimate \eqref{eq:unipotenthighermoment} in Theorem~\ref{thm:mainthm} implies the analogous statement for $a_tK$-orbits under the same assumptions:
\eqlabel{eq:compacthighermoment'}{\sup_{t>0}\int_{K}\widehat{\alpha}_{\eta,M}(a_tk;\Del_Q)^{1+\del}dm_K(k)<\infty.}
To see this, let $P$ denote the parabolic subgroup of $H$ consisting of all elements $h\in H$ such that $a_tha_{-t}$ remains bounded for $t>0$. Note that $\Gamma\cap K$ is finite, and there is a neighborhood of the identity $\cO$ in $K$ such that $K$ is covered by the sets $\cO\gamma$, where $\gamma\in \Gamma\cap K$, and the following holds: For each $\gamma\in \Gamma\cap K$ there exist local diffeomorphisms $p_{\gamma}:\cO\to P$ and $r_{\gamma}:\cO\to \mathbb{R}$ such that $k=p_{\gamma}(k)u_{r_{\gamma}(k)}\gamma$ for all $k\in \cO$. Under this setup, for each $\gamma\in \Gamma\cap K$, we have:
\eq{\begin{aligned}
    \sup_{t>0}\int_{\cO\gamma}\widehat{\alpha}_{\eta,M}(a_tk;\Del_Q)^{1+\del}dm_K(k)&=\sup_{t>0}\int_{\cO}\widehat{\alpha}_{\eta,M}\big((a_tp_{\gamma}(k)a_{-t})a_tu_{r_{\gamma}(k)}\gamma;\Del_Q\big)^{1+\del}dm_K(k)\\&\ll \sup_{t>0}\int_{\cO}\widehat{\alpha}_{\eta,M}(a_tu_{r_{\gamma}(k)};\Del_Q)^{1+\del}dm_K(k)
\end{aligned}}
using the log-Lipschitz property, as $a_tp_{\gamma}(k)a_{-t}$ is bounded for $t>0$. Since the map $r_{\gamma}:K\to \mathbb{R}$ is a diffeomorphism for each $\gamma\in \Gamma\cap K$, we can deduce from the moment estimate \eqref{eq:unipotenthighermoment} that
$$\sup_{t>0}\int_{\cO}\widehat{\alpha}_{\eta,M}(a_tu_{r_{\gamma}(k)};\Del_Q)^{1+\del}dm_K(k)<\infty.$$
This establishes \eqref{eq:compacthighermoment'}.

By a similar argument, Proposition~\ref{prop:quasinullcontribution} implies that
\eqlabel{eq:compactquasinullcontribution}{\lim_{t\to\infty}\int_{K}\sup_{v\in \Del_Q\cap (\cH_{\eta,M}\setminus\cH)}\!\!\!\!\!\!\!\!\|a_tkv\|^{-1}+\!\!\!\!\!\!\sup_{v\in \Del_Q^{*}\cap (\cH_{\eta,M}\setminus\cH)}\!\!\!\!\!\!\!\!\|a_t^*k^*v\|^{-1}dm_K(k)=0.}

\textit{Lower bound.} We first prove that
$$\displaystyle\liminf_{t\to\infty}\int_{K}\widehat{f}(a_tk;\Del_Q)\nu(k)dk\geq\int_{\bR^3} f(v)dv\int_K \nu(k)dk.$$ Let $\eps>0$. To show the lower bound, we take an approximation $f_-\in C_0^{\infty}(\bR^3\setminus\set{0})$ satisfying 
$$0\leq f_-\leq f, \qquad \int_{\bR^3}f_-(v)dv\geq (1-\eps)\int_{\bR^3}f(v)dv.$$ Then Siegel integral formula gives that
$$\int_X \widetilde{f_-}(\Del)dm_X(\Del)\geq (1-\eps)\int_{\bR^3}f(v)dv.$$
We may choose $S>1$ such that $f_-(v)=0$ for any $\|v\|\geq S$.

For $R>1$ let $X_{>R}:=\set{\Del\in X: \alpha(\Del)>R}$. Choose a continuous nonnegative function $h_R$ on $X$ satisfying $$\mathds{1}_{X_{>R+1}}\leq h_R\leq \mathds{1}_{X_{>R}}.$$
It is immediate to see that $\displaystyle\lim_{R\to\infty}\int_X h_Rdm_X=0$.

We now choose sufficiently large $R$ so that
$$\int_X \widetilde{f_-}(\Del)(1-h_R)(\Del)dm_X(\Del)\geq (1-2\eps)\int_{\bR^3}f(v)dv.$$
According to \cite{DM93}, for any indefinite irrational quadratic form $Q$ and a continuous bounded function $F$ on $X$ it holds that
$$\lim_{t\to\infty}\int_K F(a_tk\Del_Q)\nu(k)dk=\int_XFdm_X\int_K\nu(k)dk.$$
Since the function $\widetilde{f_-}(1-h_R)$ is continuous and bounded, we deduce that
\eqlabel{eq:applyingRatner}{\begin{aligned}
    \lim_{t\to\infty}\int_{K}\widetilde{f_-}&(a_tk\Del_Q)\big(1-h_R(a_tk\Del_Q)\big)\nu(k)dk\\&=\int_X \widetilde{f_-}(\Del)\big(1-h_R(\Del)\big)dm_X(\Del)\int_K\nu(k)dk\\&\geq (1-2\eps)\int_{\bR^3}f(v)dv\int_K\nu(k)dk.
\end{aligned}}

On the other hand, by a similar argument to \cite[Lemma 2]{Sch68}, there exists a constant $c'=c'(f)$ such that for any $g\in H$ and $\Del\in X$ we have
\eq{\begin{aligned}
    \left(\widetilde{f_-}(g\Del)-\widehat{f_-}(g;\Del)\right)&\big(1-h_R(g\Del)\big)=\left(\sum_{v\in\Del\cap \cH}f_-(gv)\right)\big(1-h_R(g\Del)\big)\\&\leq c'\!\!\sum_{v\in \Del_{\operatorname{prim}}\cap\cH}\!\!\!\!\|gv\|^{-1}\mathds{1}_{B(S)\setminus B(R^{-1})}(gv)\\&\quad+c'\!\!\sum_{v\in \Del_{\operatorname{prim}}^*\cap\cH}\!\!\!\!\|g^*v\|^{-1}\mathds{1}_{B(S)\setminus B(R^{-1})}(g^*v).
\end{aligned}}
Recall that there are at most eight elements in $(\Del_{\operatorname{prim}}\cup\Del_{\operatorname{prim}}^*)$ for $\Del=\Del_Q$ if $Q$ is irrational. Moreover, since the set $B(S)\setminus B(R^{-1})$ is bounded away from zero, we have
$$\lim_{t\to\infty}\int_{K}\|a_tkv\|^{-1}\mathds{1}_{B(S)\setminus B(R^{-1})}(a_tkv)dk=0,$$
$$\lim_{t\to\infty}\int_{K}\|a_t^*k^*v\|^{-1}\mathds{1}_{B(S)\setminus B(R^{-1})}(a_t^*k^*v)dk=0$$
for each $v\in \cH$, by the dominated convergence theorem. It follows that
\eqlabel{eq:f-zerocontribution}{\lim_{t\to\infty}\int_K \left(\widetilde{f_-}(a_tk\Del_Q)-\widehat{f_-}(a_tk;\Del_Q)\right)\big(1-h_R(a_tk\Del_Q)\big)\nu(k) dk=0.}
Combining \eqref{eq:applyingRatner} and \eqref{eq:f-zerocontribution}, we get
\eq{\begin{aligned}
    \lim_{t\to\infty} \int_K\widehat{f_-}&(a_tk;\Del_Q)\big(1-h_R(a_tk\Del_Q)\big)\nu(k)dk\\&=\lim_{t\to\infty} \int_K\widetilde{f_-}(a_tk;\Del_Q)\big(1-h_R(a_tk\Del_Q)\big)\nu(k)dk\\&\geq (1-2\eps)\int_{\bR^3}f(v)dv\int_K\nu(k)dk
\end{aligned}}
for any $\eps>0$. Therefore, taking $\eps\to0$ we conclude that
$$\liminf_{t\to\infty}\int_K \widehat{f}(a_tk;\Del_Q)\nu(k) dk\geq \liminf_{t\to\infty}\int_K \widehat{f_-}(a_tk;\Del_Q)\nu(k) dk\geq \int_{\bR^3}f(v)dv\int_K\nu(k)dk.$$

\textit{Upper bound.} We now prove that
$$\displaystyle\limsup_{t\to\infty}\int_{K}\widehat{f}(a_tk;\Del_Q)\nu(k)dk\leq\int_{\bR^3}f(v)dv\int_K\nu(k)dk.$$ Let $\eps>0$. To show the upper bound, we take an approximation $f_+\in C_0^{\infty}(\bR^3\setminus\set{0})$ satisfying 
$$0\leq f\leq f_+, \qquad \int_{\bR^3}f_+dm_{\bR^3}\leq (1+\eps)\int_{\bR^3}f(v)dv.$$ Then Siegel integral formula gives that
$$\int_X \widetilde{f_+}(\Del)dm_X(\Del)\leq (1+\eps)\int_{\bR^3}f(v)dv.$$
Then we have
\eqlabel{eq:f+Ratner}{\begin{aligned}
    \lim_{t\to\infty}\int_{K}\widehat{f_+}(a_tk;\Del_Q)&\big(1-h_R(a_tk\Del_Q)\big)\nu(k)dk\\&\leq\lim_{t\to\infty}\int_{K}\widetilde{f_+}(a_tk\Del_Q)\big(1-h_R(a_tk\Del_Q)\big)\nu(k)dk\\&\leq\int_X \widetilde{f_+}(\Del)\big(1-h_R(\Del)\big)dm_X(\Del)\int_K\nu(k)dk\\&\leq (1+\eps)\int_{\bR^3}f(v)dv\int_K\nu(k)dk
\end{aligned}}
for any $R>1$. We shall now estimate $\int_K \widehat{f_+}(a_tk;\Del_Q)h_R(a_tk\Del_Q)\nu(k)dk$. Note that
\eqlabel{eq:f+split}{\begin{aligned}
    \widehat{f_+}(a_tk;\Del_Q)&\leq  \widehat{(f_+)}_{\eta,M}(a_tk;\Del_Q)+\|f\|_{\infty}\left(\sup_{v\in \Del_Q\cap (\cH_{\eta,M}\setminus\cH)}\!\!\!\!\!\!\!\!\|a_tkv\|^{-1}+\!\!\!\!\sup_{v\in \Del_Q^{*}\cap (\cH_{\eta,M}\setminus\cH)}\!\!\!\!\!\!\!\!\|a_t^*kv\|^{-1}\right)\\&\leq c\widehat{\alpha}_{\eta,M}(a_tk;\Del_Q)+\|f\|_{\infty}\left(\sup_{v\in \Del_Q\cap (\cH_{\eta,M}\setminus\cH)}\!\!\!\!\!\!\!\!\|a_tkv\|^{-1}+\!\!\!\!\sup_{v\in \Del_Q^{*}\cap (\cH_{\eta,M}\setminus\cH)}\!\!\!\!\!\!\!\!\|a_t^*kv\|^{-1}\right)
\end{aligned}}
holds by the Lipschitz principle in Lemma~\ref{lem:Lipschitzprinciple}. 

For the first term, let us write
\eqlabel{eq:auxillaryalpha}{\begin{aligned}
    \int_K &\widehat{\alpha}_{\eta,M}(a_tk;\Del_Q)h_R(a_tk\Del_Q)\nu(k)dk\\&\leq \left(\sup_{k\in K}|\nu(k)|\right)\int_K \big(\widehat{\alpha}_{\eta,M}(a_tk;\Del_Q)+\alpha(a_tk\Del_Q)^{\frac{1}{2}}\big) h_R(a_tk\Del_Q)dk.
\end{aligned}}
Observe that $\big(\widehat{\alpha}_{\eta,M}(a_tk;\Del_Q)+\alpha(a_tk\Del_Q)^{\frac{1}{2}}\big) h_R(a_tk\Del_Q)$ is either zero or at least $R^{\frac{1}{2}}$. Thus \eqref{eq:auxillaryalpha} is bounded above by
\eqlabel{eq:genericupperbound}{\begin{aligned}
    &\leq \left(\sup_{k\in K}|\nu(k)|\right)\int_K\big(\widehat{\alpha}_{\eta,M}(a_tk;\Del_Q)+\alpha(a_tk\Del_Q)^{\frac{1}{2}}\big)^{1+\del}R^{-\frac{\del}{2}}dk\\&\ll \left(\sup_{k\in K}|\nu(k)|\right)R^{-\frac{\del}{2}}\int_K\big(\widehat{\alpha}_{\eta,M}(a_tk;\Del_Q)^{1+\del}+\alpha(a_tk\Del_Q)^{\frac{1+\del}{2}})dk\\&\ll \left(\sup_{k\in K}|\nu(k)|\right)R^{-\frac{\del}{2}}.
\end{aligned}}
In the last line, the moment estimates \eqref{eq:compacthighermoment'} and \eqref{MomentEMM} are used. For the remaining terms in \eqref{eq:f+split}, we can deduce from \eqref{eq:compactquasinullcontribution} (contribution for quasi-null vectors) that
\eqlabel{eq:quasinullcontribution}{\lim_{t\to\infty}\int_{K}\sup_{v\in \Del_Q\cap (\cH_{\eta,M}\setminus\cH)}\!\!\!\!\!\!\|a_tkv\|^{-1}+\!\!\!\!\sup_{v\in \Del_Q^{*}\cap (\cH_{\eta,M}\setminus\cH)}\!\!\!\!\!\!\|a_t^*kv\|^{-1}dk=0.}

Combining \eqref{eq:f+Ratner}, \eqref{eq:f+split}, \eqref{eq:genericupperbound}, and \eqref{eq:quasinullcontribution} altogether, we get
$$\limsup_{t\to\infty}\int_{K}\widehat{f_+}(a_tk;\Del_Q)\nu(k)dk\leq (1+\eps)\int_{\bR^3}f(v)dv\int_K\nu(k)dk+O(R^{-\frac{\del}{2}}).$$
Taking $R\to\infty$ and $\eps\to0$, we conclude that
$$\limsup_{t\to\infty}\int_{K}\widehat{f}(a_tk;\Del_Q)\nu(k)dk\leq \int_{\bR^3}f(v)dv\int_K\nu(k)dk.$$

\subsection{Proof of Theorem~\ref{thm:quantitativeoppenheim}}\label{subsec:exceptionalsubspaces}
To deduce the full count in Theorem~\ref{thm:quantitativeoppenheim} from the modified count in Theorem \ref{thm:quantitativeoppenheim0}, it suffices to count the number of points on isotropic rational lines and planes. 

\begin{lem}\label{lem:degenerateplane}
    For a ternary quadratic form $Q$ and a plane $0\in P\subset \mathbb{R}^3$ the restriction $Q|_P$ is the square of a linear form if and only if the orthogonal line to $P$ is isotropic for the dual form $Q^*$.
\end{lem}
\begin{proof}
    We first claim that $Q(\bv_1)Q(\bv_2)-Q(\bv_1,\bv_2)^2=Q^*(\bv_1\times\bv_2)$ holds for any $\bv_1,\bv_2\in\mathbb{R}^3$. The case $\bv_1\times\bv_2=0$ is trivial. Otherwise we may choose $g\in G$ so that $g\bv_1=\be_1$ and $g\bv_2=\be_2$. Then $g^*(\bv_1\times \bv_2)=(g\bv_1\times g\bv_2)=\be_3$ follows. Let $Q'=Q\circ g^{-1}$. By a straightforward calculation we can check that $Q'(\be_1)Q'(\be_1)-Q'(\be_1,\be_2)^2=(Q')^*(\be_3)$ holds for any indefinite ternary quadratic form $Q'$ with $\operatorname{det}(Q')=1$. We thus have
\begin{equation}\label{eq:outerproductduality}
    \begin{aligned}
        Q(\bv_1)Q(\bv_2)-Q(\bv_1,\bv_2)^2&=Q'(\be_1)Q'(\be_2)-Q'(\be_1,\be_2)^2\\&=(Q')^*(\be_3)=(Q')^*\big(g^*(\bv_1\times\bv_2)\big)=Q^*(\bv_1\times\bv_2).
    \end{aligned}
\end{equation}

Let us choose a basis $\set{\bv_1,\bv_2}$ of $P$. Then for any $\bv=a_1\bv_1+a_2\bv_2\in P$ with $a_1,a_2\in\mathbb{R}$ we have
$$Q(\bv)=Q(a_1\bv_1+a_2\bv_2)=Q(\bv_1)a_1^2+2Q(\bv_1,\bv_2)a_1a_2+Q(\bv_2)a_2^2,$$
hence $Q|_P$ is the square of a linear form if and only if $Q(\bv_1)Q(\bv_2)-Q(\bv_1,\bv_2)^2=0$. In view \eqref{eq:outerproductduality} this completes the proof.
\end{proof}

For a given indefinite ternary quadratic form $Q$, let us denote 
$$\cL(Q):=\set{\bm\in\bZ^3_{\operatorname{prim}}: Q(\bm)=0},$$
$$\cP(Q):=\set{\bm\in\bZ^3_{\operatorname{prim}}: Q^*(\bm)=0}.$$

Each $\bm\in \cL(Q)$ is on an isotropic rational line through the origin, and each $\bm\in \cP(Q)$ is orthogonal to an isotropic rational plane through the origin. Since there are two primitive vectors on a rational line, Lemma \ref{lem:fourisotropic} implies that $|\cL(Q)|$ and $|\cP(Q)|$ are at most $8$ if $Q$ is irrational.

The total number of the points of norm $<T$ on the isotropic rational lines is asymptotically $\mathsf{L}_QT$, where
$$\mathsf{L}_Q:=\sum_{\bm\in\cL(Q)}\frac{1}{\|\bm\|}.$$
Since the value of $Q$ is always zero on isotropic rational lines, the contribution of the points on isotropic rational lines is $\mathsf{L}_Q\mathds{1}_{(a,b)}(0)$.

Now we estimate the contribution of points on isotropic rational planes. Let $\bm\in \cP(Q)$. We choose an integral basis $\set{\bn_1=\bn_1(\bm),\bn_2=\bn_2(\bm)}$ of the plane orthogonal to $\bm$. Then $\bn_1\times \bn_2$ is a scalar multiple of $\bm$.

Using the identity \eqref{eq:outerproductduality}, we find that
$$Q(\bn_1)Q(\bn_2)-Q(\bn_1,\bn_2)^2=Q^*(\bn_1\times\bn_2)=0,$$
hence $\frac{Q(\bn_1,\bn_2)}{Q(\bn_1)}=\frac{Q(\bn_2
)}{Q(\bn_1,\bn_2)}$ holds. Note that $Q(\bn_1)$ and $Q(\bn_2)$ are positive as $Q$ is of signature $(2,1)$.


We shall count the number of points $v$ in this plane that satisfy
$\|v\|<T$ and $a<Q(v)<b$. Write $v=k_1\bn_1+k_2\bn_2$ with $k_1,k_2\in\bZ$. Then
\eq{\begin{aligned}
Q(v)&=Q(k_1\bn_1+k_2\bn_2)\\&=Q(\bn_1)k_1^2+2Q(\bn_1,\bn_2)k_1k_2+Q(\bn_2)k_2^2\\&=Q(\bn_1)\left(k_1+\frac{Q(\bn_1,\bn_2)}{Q(\bn_1)}k_2\right)^2+\frac{Q(\bn_1)Q(\bn_2)-Q(\bn_1,\bn_2)^2}{Q(\bn_1)}k_2^2\\&=Q(\bn_1)\left(k_1+\frac{Q(\bn_1,\bn_2)}{Q(\bn_1)}k_2\right)^2=\left(\sqrt{Q(\bn_1)}k_1+\sqrt{Q(\bn_2)}k_2\right)^2.
\end{aligned}}
Hence, the system of the inequalities $\|v\|<T$ and $a<Q(v)<b$ is written
\begin{equation}\label{eq:systemofinequalities}
    \begin{aligned}
        &\quad\quad\|k_1\bn_1+k_2\bn_2\|<T,\\
        \sqrt{a^+}<&\left|\sqrt{Q(\bn_1)}k_1+\sqrt{Q(\bn_2)}k_2\right|<\sqrt{b^+}.
    \end{aligned}
\end{equation}
One can calculate that the area of this region is asymptotically
\eqlabel{eq:volumecalculation}{\frac{2(\sqrt{b^+}-\sqrt{a^+})}{\|\sqrt{Q(\bn_1)}\bn_2-\sqrt{Q(\bn_2)}\bn_1\|}T \quad\textrm{ as }\; T\to\infty,}
where $a^+:=\max(a,0)$ and $b^+:=\max(b,0)$.

If $\frac{\sqrt{Q(\bn_2)}}{\sqrt{Q(\bn_1)}}=\frac{Q(\bn_1,\bn_2)}{Q(\bn_1)}$ is irrational, the number of integral points $(k_1,k_2)$ in the region is asymptotically the same as \eqref{eq:volumecalculation}, hence linear in $T$. If $\frac{\sqrt{Q(\bn_2)}}{\sqrt{Q(\bn_1)}}=\frac{Q(\bn_1,\bn_2)}{Q(\bn_1)}$ is rational, then for the map $(k_1,k_2)\mapsto \sqrt{Q(\bn_1)}k_1+\sqrt{Q(\bn_2)}k_2$, the preimage of the interval $(\sqrt{a^+},\sqrt{b^+})$ consists of finitely many parallel rational lines in $\mathbb{Z}^2$. Consequently, the number of solutions to the system of inequalities \eqref{eq:systemofinequalities} remains linear in $T$ in this case as well.


To summarize, we demonstrated that for any $a<b$ the number of points $v\in\mathbb{Z}^3$ on isotropic lines and planes satisfying $\|v\|<T$ and $a<Q(v)<b$ is asymptotically $\mathsf{I}_{Q}(a,b)T$ as $T\to\infty$, where the constant $\mathsf{I}_{Q}(a,b)\geq0$ depends only on $a,b$, and $Q$. In combination with Theorem \ref{thm:quantitativeoppenheim0}, we conclude that
\eq{\begin{aligned}
    \lim_{T\to\infty}\frac{N_Q(a,b,T)}{T}=\lim_{T\to\infty}\frac{\widetilde{N}_Q(a,b,T)+\mathsf{I}_{Q}(a,b)T}{T}=\mathsf{C}_Q(b-a)+\mathsf{I}_{Q}(a,b).
\end{aligned}}


\def\cprime{$'$} \def\cprime{$'$} \def\cprime{$'$}
\providecommand{\bysame}{\leavevmode\hbox to3em{\hrulefill}\thinspace}
\providecommand{\MR}{\relax\ifhmode\unskip\space\fi MR }
\providecommand{\MRhref}[2]{%
  \href{http://www.ams.org/mathscinet-getitem?mr=#1}{#2}
}

\end{document}